\RequirePackage{fix-cm}
\pdfoutput=1
\let\rextraRefstepcounter\refstepcounter
\let\rextraLabel\label
\AddToHook{package/hyperref/before}{%
  \let\refstepcounter\rextraRefstepcounter
  \let\label\rextraLabel}
\documentclass[onefignum,onetabnum,onealgnum,final,hidelinks]{siamart220329}
\usepackage[left=4cm,right=4cm]{geometry}
\usepackage{amssymb}
\usepackage{microtype}
\usepackage[numbers,sort&compress]{natbib}
\bibpunct[, ]{[}{]}{,}{n}{}{,}%
\renewtheorem{lemma}{Lemma}[section]
\newtheorem{assumption}{Assumption}[section]
\theorembodyfont{\normalfont}

\renewtheorem{definition}{Definition}[section]
\usepackage{epstopdf}
\usepackage{mathtools}
\usepackage{enumerate,enumitem}
\usepackage{verbatim}
\usepackage{mathrsfs}
\DeclareFontFamily{U}{rsfs}{\skewchar\font127}
\DeclareFontShape{U}{rsfs}{m}{n}{<-6>rsfs5<6-8>rsfs7<8->rsfs10}{}
\usepackage{booktabs}
\usepackage{bm}
\usepackage{algorithm}
\usepackage{algorithmic}

\let\orignewcommand\newcommand
\let\newcommand\providecommand

\newcommand{\ba}{\begin{array}}
\newcommand{\ea}{\end{array}}

\providecommand{\bit}{\begin{itemize}}
\providecommand{\eit}{\end{itemize}}
\providecommand{\be}{\begin{equation}}
\providecommand{\ee}{\end{equation}}
\providecommand{\bee}{\begin{equation*}}
\providecommand{\eee}{\end{equation*}}
\providecommand{\bea}{\begin{eqnarray}}
\providecommand{\eea}{\end{eqnarray}}

\newcommand{\st}{\mathrm{s.t.}}

\newcommand{\by}{\mathrm{\bf y}}
\newcommand{\bs}{\mathrm{\bf s}}
\newcommand{\bg}{\mathrm{\bf g}}

\newcommand{\bx}{\mathbf{x}}

\newcommand{\hbg}{\widehat{\mathbf{g}}}

\newcommand{\bW}{\mathbf{W}}
\newcommand{\bV}{\mathbf{V}}

\newcommand{\Rmn}[1]{\uppercase\expandafter{\romannumeral#1}}
\renewcommand{\theenumi}{\roman{enumi}}

\newcommand{\Rcal}{\mathcal{R}}

\newcommand{\Pcal}{\mathcal{P}}

\newcommand{\Mcal}{\mathcal{M}}

\newcommand{\grad}{\mathrm{grad}}

\newcommand{\R}{\mathbb{R}}

\newcommand{\iprod}[2]{\left \langle #1, #2 \right \rangle }

\let\newcommand\orignewcommand

\usepackage{textcomp}
\usepackage{indentfirst}
\hypersetup{%
  pdftitle={Riemannian EXTRAs: Extending EXTRA to decentralized manifold optimization by incorporating normal correction for provable convergence},
  pdfauthor={Jiayuan Wu, Zhanwang Deng, Jiang Hu, Weijie Su, and Zaiwen Wen},
  pdfsubject={Decentralized Riemannian optimization}}
\ifpdf
  \DeclareGraphicsExtensions{.eps,.pdf,.png,.jpg}
\else
  \DeclareGraphicsExtensions{.eps}
\fi

\makeatletter
\def\refstepcounter@optarg[#1]#2{%
  \refstepcounter@noarg{#2}%
  \cref@constructprefix{#2}{\cref@result}%
  \protected@edef\cref@currentlabel{%
    [#1][\arabic{#2}][\cref@result]%
    \csname p@#2\endcsname\csname the#2\endcsname}}
\makeatother

\headers{Riemannian EXTRAs}{J. Wu, Z. Deng, J. Hu, W. Su, Z. Wen}
\title{Riemannian EXTRAs: Extending EXTRA to decentralized manifold optimization by incorporating normal correction for provable convergence}
\author{%
Jiayuan Wu\thanks{\raggedright Wharton Statistics and Data Science Department, University of Pennsylvania, Philadelphia, PA 19104, U.S. (\email{jyuanw@upenn.edu}).\endgraf}
\and Zhanwang Deng\thanks{\raggedright Academy for Advanced Interdisciplinary Studies, Peking University, Beijing 100871, China (\email{dzw\_opt2022@stu.pku.edu.cn}).\endgraf}
\and Jiang Hu\thanks{\raggedright Yau Mathematical Sciences Center, Tsinghua University, Beijing 100084, China (\email{jianghu@tsinghua.edu.cn}).\endgraf}
\and Weijie Su\thanks{\raggedright Wharton Statistics and Data Science Department, University of Pennsylvania, Philadelphia, PA 19104, U.S. (\email{suw@wharton.upenn.edu}).\endgraf}
\and Zaiwen Wen\thanks{\raggedright Beijing International Center for Mathematical Research and Center for Machine Learning Research, Peking University, Beijing 100871, China (\email{wenzw@pku.edu.cn}).\endgraf}
}

\allowdisplaybreaks[2]
\begin{document}
\raggedbottom
\maketitle

\begin{abstract}
{We consider decentralized optimization over a compact Riemannian submanifold in a network of $n$ nodes, where each node $i$ owns a smooth nonconvex local objective $f_i$ defined by its private data, and the goal is to minimize the global objective $F(x)=\frac{1}{n}\sum_{i=1}^n f_i(x)$ using only local communications. A naive Riemannian extension of EXTRA leaves the auxiliary normal-space error uncontrolled and does not inherit the Euclidean convergence guarantees. We develop projection-based Riemannian EXTRA (REXTRA) and a retraction-based counterpart that incorporate a normal correction into both the communicated variable and the auxiliary update, while using only one local-variable-sized message per optimization iteration. By jointly controlling consensus, tangent, and normal errors, our analysis interprets the update as a network-based Riemannian inexact-gradient step. For projection-based REXTRA, a sufficiently small constant stepsize and a local initialization yield vanishing consensus, stationary accumulation points, and an $\mathcal O(1/k)$ best-iterate stationarity--consensus bound. For the retraction-based method, we prove that its difference from the projection update is uniformly second order and establish the same convergence properties and rate under a local initialization. Numerical experiments on representative manifold-valued learning tasks demonstrate the communication efficiency of the REXTRA updates.}
\end{abstract}

\section{Introduction}
{Decentralized optimization has received considerable attention because of its relevance to large-scale distributed systems, including applications in distributed computing, machine learning, control, and signal processing. In these systems, data is often distributed across multiple computational nodes, making centralized approaches impractical because of limitations in storage, communication, and computation. In this work, we focus on decentralized smooth optimization problems constrained to a compact Riemannian submanifold embedded in Euclidean space. Specifically, we consider the following formulation:}
\be \label{prob:original}
\begin{aligned}
  \min \quad & \frac{1}{n}\sum_{i=1}^n f_i(x_i), \\
  \st \quad & x_1 = \cdots = x_n,\quad x_i \in \Mcal,\quad \forall i=1,\dots,n,
\end{aligned}
\ee
where \( n \) denotes the number of nodes, each function \( f_i \) is a smooth local objective associated with the data at node \( i \), and \( \mathcal{M} \subset \mathbb{R}^{d \times r} \) is a compact smooth manifold, such as the Stiefel manifold \( {\rm St}(d, r) := \{ x \in \mathbb{R}^{d \times r} : x^\top x = I_r \} \). This class of problems naturally arises in various machine learning and signal processing tasks, including principal component analysis \cite{ye2021deepca}, low-rank matrix completion \cite{boumal2015low,hu2022riemannian}, neural networks with batch normalization \cite{cho2017riemannian,hu2022riemannian}, networks with orthogonality constraints \cite{arjovsky2016unitary,vorontsov2017orthogonality,huang2018orthogonal,eryilmaz2022understanding}, and parameter-efficient fine-tuning of large models \cite{zhangriemannian,hu2024retraction}.

\subsection{Related work}
{Decentralized optimization in Euclidean space (i.e., $\mathcal{M} = \mathbb{R}^{d \times r}$) has been extensively studied over the past few decades. A seminal method is decentralized (sub)gradient descent (DGD) \cite{nedic2009distributed,tsitsiklis1986distributed,yuan2016convergence}, which combines local gradient updates with consensus steps by communicating local iterates. However, DGD with constant stepsizes can only converge to a neighborhood of a stationary point and cannot guarantee exact convergence \cite{yuan2016convergence}. To address this limitation, two main types of methods have been developed. The first type, known as gradient-tracking methods, communicates both local iterates and directions to track the centralized gradient. The second type uses primal--dual updates to achieve exact convergence. Notable examples include EXTRA \cite{shi2015extra}, {DLM} \cite{ling2015dlm}, exact diffusion \cite{yuan2018exact}, and NIDS \cite{li2019decentralized}. {Compared with gradient tracking, EXTRA and its variants communicate only local iterates, reducing the payload by 50\% per communication round when both methods use messages of the same dimension over the same network.} Although EXTRA was originally designed for convex objectives, extensions to nonconvex problems have been studied in \cite{bianchi2012convergence,di2016next,hong2017prox,tatarenko2017non,scutari2019distributed,sun2020improving,wai2017decentralized,zeng2018nonconvex}. A typical treatment of nonconvexity relies on the primal--dual analysis of augmented Lagrangian methods introduced in \cite{hong2017prox}, whose key step is to bound the change in the dual variable by changes in the primal variables through the optimality conditions of the primal subproblems. This analysis does not extend directly to manifold-constrained settings because the tangent space in the optimality conditions varies with the primal iterates, making the difference between consecutive dual updates difficult to characterize. Existing results on manifold inexact augmented Lagrangian methods \cite{deng2025oracle} typically require increasingly accurate solutions of the primal subproblem and a diverging penalty parameter. This is one of the main challenges in extending EXTRA-type methods from Euclidean space to Riemannian manifolds.}

{Extending decentralized optimization algorithms from Euclidean spaces to Riemannian manifolds introduces additional difficulties caused by the nonconvexity and nonlinearity of the manifold constraints. In particular, standard consensus updates based on weighted Euclidean averages of local iterates are generally infeasible on manifolds. To address this issue, \cite{tron2012riemannian} proposes a Riemannian consensus algorithm based on geodesic distances, but the method requires expensive operations such as exponential maps and vector transports. To design computationally affordable methods, \cite{chen2021local} studies consensus over the Stiefel manifold using Euclidean distances and proves linear convergence of a Riemannian gradient method under a restricted strong convexity condition and multiple consensus steps. A decentralized Riemannian gradient-tracking method with exact convergence under constant stepsizes was subsequently developed in \cite{chen2021decentralized}. Later, \cite{deng2023decentralized} proposed a projected Riemannian gradient method with gradient tracking for general compact manifolds. Its analysis uses \textit{proximal smoothness}, which ensures that the metric projection onto the manifold is single-valued and Lipschitz continuous in a neighborhood of the manifold. Following these works, many recent studies have sought to improve the communication and computational efficiency of decentralized manifold optimization. Examples include decentralized Riemannian conjugate-gradient methods \cite{chen2023decentralizediclr}, decentralized natural-gradient methods \cite{hu2023decentralized}, decentralized manifold proximal-gradient methods \cite{wang2025decentralized}, and decentralized gradient methods with communication compression \cite{hu2024improving}. Another line of work focuses on retraction-free methods. An early example is \cite{wang2022decentralized}, which uses an augmented Lagrangian approach to relax the manifold constraint through penalization. This idea has been further developed in \cite{sun2024global,sun2024local}.}
{However, many existing decentralized Riemannian algorithms require exchanging both local iterates and local directions, or performing multiple consensus rounds, to guarantee exact convergence with constant stepsizes.}

\begin{table}[H]
\centering
\caption{Comparison of distributed manifold optimization algorithms.}
\label{tab:algorithm_comparison}
\setlength{\tabcolsep}{2.5pt}
\resizebox{\textwidth}{!}{
\begin{tabular}{lccccc}
\toprule
\textbf{Algorithm} & \textbf{Commu. quan.}$^\diamond$ & \textbf{Commu. rounds}$^\dagger$ & \textbf{Step size}$^\triangleright$ & \textbf{Rate} & \textbf{Manifold}$^\ddagger$ \\
\midrule
DPRGD~\cite{deng2023decentralized} & local variable & {$\mathcal{O}(\log n/(-\log\sigma_2))$} & Diminishing & $\mathcal{O}(1/\sqrt{k})$ & Compact \\
DPRGT~\cite{deng2023decentralized} & (local variable, local direction) & {$\mathcal{O}(\log n/(-\log\sigma_2))$} & Constant & $\mathcal{O}(1/k)$ & Compact \\
DRDGD~\cite{chen2021decentralized} & local variable & {$\mathcal{O}(\log n/(-\log\sigma_2))$} & Diminishing & $\mathcal{O}(1/\sqrt{k})$ & Stiefel \\
DRGTA~\cite{chen2021decentralized} & (local variable, local direction) & {$\mathcal{O}(\log n/(-\log\sigma_2))$} & Constant & $\mathcal{O}(1/k)$ & Stiefel \\
  DRCGD~\cite{chen2023decentralizediclr} & (local variable, local direction) & {$\mathcal{O}(\log n/(-\log\sigma_2))$} & Diminishing$^\clubsuit$ & not given & Stiefel \\
  {\textbf{REXTRA}$^\S$} & \textbf{local variable} & \textbf{1} & \textbf{Constant} & {$\mathcal{O}(1/k)$} & \textbf{Compact}  \\
\bottomrule
\end{tabular}}
\begin{flushleft}
\footnotesize
{$^\diamond$ Communication volume per iteration.}\\
{$^\dagger$ Communication rounds per iteration. For the multi-step baselines, the displayed bound is the number of mixing steps required to make $\sigma_2^t=\mathcal O(1/n)$ when $0<\sigma_2<1$; when $\sigma_2=0$, one mixing step gives exact averaging.}  \\
{$^\clubsuit$ DRCGD also requires the step to satisfy specified Armijo and strong Wolfe conditions.} \\
{$^\triangleright$ This column records whether the cited convergence result uses a fixed or diminishing stepsize.} \\
{$^\ddagger$ The Stiefel manifold is a particular compact manifold. Rates are reported in the stationarity measures used by the cited papers and need not use identical normalizations.}
{$^\S$ REXTRA denotes the projection-based method in Algorithm~\ref{alg}; its local best-iterate rate is proved in Section~\ref{sec-con}.}
\end{flushleft}
\end{table}

\subsection{Contributions}
The goal of this paper is to develop communication-efficient decentralized methods with compact manifold constraints that address data heterogeneity. Our contributions are summarized as follows:

\begin{itemize}
\item \textbf{Communication-efficient decentralized Riemannian algorithms.}
{We develop projection-based REXTRA and a retraction-based counterpart by adding a normal correction to the primal EXTRA recursion \cite{shi2015extra}. The correction feeds the auxiliary normal component into both the communicated variable and the auxiliary update. Both methods communicate only one local-variable-sized matrix per optimization iteration in the presence of data heterogeneity. The retraction variant preserves the primal correction and communication pattern of Algorithm~\ref{alg}, replacing only its final metric projection with a tangent-space retraction. Experiments on decentralized principal component analysis, generalized eigenvalue problems, and low-rank matrix completion show that REXTRA reaches high-accuracy solutions with fewer communicated scalar entries than the tested baselines. They also demonstrate that REXTRA supports larger stepsizes and can converge faster in these settings.}

\item {\textbf{Convergence guarantees.}}
{The Euclidean analysis based on the augmented Lagrangian method (ALM) bounds dual increments through primal optimality conditions. On a manifold, these conditions provide only tangential information and do not control the normal-space error as tangent spaces vary. Thus the Euclidean ALM argument does not establish convergence of a naive Riemannian extension. Our analysis exploits the normal correction to jointly control consensus, tangent, and normal errors directly through the primal recursions. This yields a network-based Riemannian inexact-gradient interpretation and a $\mathcal O(1/k)$ best-iterate bound for projection-based REXTRA. For retraction-based REXTRA, a uniformly controlled second-order comparison with the projection update yields the same convergence guarantees and rate under the stated local initialization conditions.}
\end{itemize}

\subsection{Notation}
For a compact submanifold $\Mcal$ of $\mathbb{R}^{d \times r}$, we adopt the Euclidean inner product $\iprod{\cdot}{\cdot}$ as the Riemannian metric and use $\|\cdot\|$ to denote the associated Euclidean norm. The $n$-fold Cartesian product of $\Mcal$ is denoted as $\Mcal^n = \Mcal \times \cdots \times \Mcal$. For any $x \in \Mcal$, we denote the tangent space and normal space at $x$ by $T_x\Mcal$ and $N_x\Mcal$, respectively. For a differentiable function $h: \R^{d\times r} \to \R$, we denote its Euclidean gradient by $\nabla h(x)$ and its Riemannian gradient by $\grad h(x)$. For a positive integer $n$, we define $[n] := \{1, \dots, n\}$. Let $\mathbf{1}_n \in \mathbb{R}^n$ denote the vector of all ones, and define $J := \frac{1}{n} \mathbf{1}_n \mathbf{1}_n^\top$.

We introduce the notation conventions for boldface variables. Taking $x$ as an example, we denote $x_i$ as the local variable at the $i$-th node, and define the Euclidean average as $\hat{x} := \frac{1}{n} \sum_{i=1}^n x_i$. We use the boldface notations: $
\bx := [x_1^\top, \dots, x_n^\top]^\top \in \R^{(nd) \times r}$, where $\bx$ stacks all local variables. In the context of iterative algorithms, we use $x_{i,k}$ to denote the local variable at the $i$-th node at iteration $k$, and define $\hat{x}_k := \frac{1}{n} \sum_{i=1}^n x_{i,k}$. Similarly, we introduce
$
\bx_k := [x_{1,k}^\top, \dots, x_{n,k}^\top]^\top \in \R^{(nd) \times r}$. {Whenever the metric projection of an average is single-valued, we write $\bar{x}:=\Pcal_{\Mcal}(\hat{x})$ and $\bar{\bx}:=\mathbf 1_n\otimes\bar{x}$; similarly, $\bar{x}_k:=\Pcal_{\Mcal}(\hat{x}_k)$ and $\bar{\bx}_k:=\mathbf 1_n\otimes\bar{x}_k$. For the global consensus objective, we write $F(z):=\frac1n\sum_{i=1}^n f_i(z)$. For the stacked EXTRA recursion, we write $f(\bx):=\sum_{i=1}^n f_i(x_i)$; on the consensus subspace, $f(\mathbf1_n\otimes z)=nF(z)$. We also use $\mathbf J:=J\otimes I_d$ for the lifted averaging matrix, and $I$ denotes an identity matrix of compatible dimension when the size is clear.} We set $\bW := W \otimes I_d \in \R^{nd \times nd}$, where $\otimes$ denotes the Kronecker product. {The notation $\widehat T=\mathcal O(T)$ means that $\widehat T\le cT$ for some constant $c>0$ independent of the relevant iteration index.}

\section{\texorpdfstring{Preliminaries}{Preliminaries}}
In this section, we introduce basic concepts related to decentralized manifold optimization.

\subsection{Compact submanifolds and smoothness of the projection operator}
{Compact smooth embedded submanifolds of Euclidean space, as described in~\cite[Section 3.3]{absil2009optimization}, inherit the subspace topology from their ambient Euclidean space. Throughout this work, we focus on compact smooth submanifolds, including the (generalized) Stiefel, Grassmann, and oblique manifolds.}
In the context of designing and analyzing decentralized optimization algorithms over manifolds, it has been recognized in \cite{deng2023decentralized} that the smoothness properties of the projection operator onto $\Mcal$ play a critical role, particularly within certain neighborhoods. To set the stage, we first introduce the notion of proximal smoothness. {Given a point $x \in \mathbb{R}^{d \times r}$, we define its distance to $\Mcal$ and its nearest-point projection onto $\Mcal$ as}
 $
 {\operatorname{dist}(x,\Mcal):=\inf_{y\in\Mcal}\|y-x\|,\;\;
 \Pcal_{\Mcal}(x):=\operatorname*{argmin}_{y\in\Mcal}\|y-x\|,}
 $
 respectively. For any real number $\tau >0$, we define the $\tau$-tube around $\mathcal{M}$ as $U_{\mathcal{M}}(\tau): = \{x:~\text{dist}(x,\mathcal{M}) < \tau \}$, and $\bar{U}_{\Mcal}(\tau):= \{x:~\text{dist}(x,\mathcal{M}) \leq \tau\}$.
{A closed set $\mathcal{M}$ is said to be $R$-proximally smooth if the projection $\Pcal_{\mathcal{M}}(x)$ is a singleton whenever $\text{dist}(x,\mathcal{M}) < R$. Following \cite{clarke1995proximal}, an $R$-proximally smooth set $\mathcal{M}$ has the following two properties. First, for every $\tau\in(0,R)$,}
    \be \label{eq:lip-proj}
    \left\| \Pcal_{\mathcal{M}} (x) -\Pcal_{\mathcal{M}} (y)\right\| \leq \frac{R}{R-\tau}\|x - y\|,~~ \forall x,y \in \bar{U}_{\mathcal{M}}(\tau).
    \ee
{In particular, $\Pcal_{\Mcal}$ is asymptotically $1$-Lipschitz as $\tau\to0$. Second, for every $x\in\mathcal M$, $v\in N_x\Mcal$, and $y\in\Mcal$,}
    \be \label{proximally-0}
    \iprod{v}{y-x} \leq \frac{\|v\|}{2R}\|y-x\|^2.
    \ee
{The estimates \eqref{eq:lip-proj} and \eqref{proximally-0} are standard consequences of proximal smoothness and underlie the local projection and chord estimates used below.}

It is known that any compact $C^2$ submanifold of Euclidean space is proximally smooth \cite{clarke1995proximal,balashov2021gradient,davis2020stochastic}. For instance, the Stiefel manifold is shown to be a $1$-proximally smooth set \cite{balashov2021gradient}. The asymptotic 1-Lipschitz property of the projection operator $\Pcal_{\Mcal}$ plays an important role in achieving linear convergence rates under single-step consensus schemes. Throughout this paper, we assume that the manifold $\Mcal$ in problem~\eqref{prob:original} is $R$-proximally smooth for some $R>0$.

{Proximal smoothness ensures the Lipschitz continuity of $\Pcal_{\Mcal}$ but does not by itself guarantee differentiability or higher-order smoothness. The smooth embedded structure of $\Mcal$ further implies that $\Pcal_{\Mcal}$ is smooth on a sufficiently small tube $U_{\Mcal}(\tau)$, with $0<\tau<R$ \cite[Lemma]{foote1984regularity}. At every $x\in\Mcal$, its derivative satisfies $D\Pcal_{\Mcal}(x)[\cdot]=\Pcal_{T_x\Mcal}(\cdot)$ {\cite[Lemma 4]{absil2012projection}}.}
{These lead to the following fact \cite[Lemma 4.3]{deng2023decentralized}.}
\begin{lemma}
\label{lemma-project}
{Given an $R$-proximally smooth compact submanifold $\Mcal$, there exists a constant $Q > 0$ such that, for any $x \in\Mcal$ and any $u\in \{u\in \mathbb{R}^{d\times r}:\|u\|\leq \frac{R}{2}\}$,}
\be
\label{eq:projec-second-order1}
\| \Pcal_{\Mcal}(x + u)  - x - \Pcal_{T_x\Mcal}(u) \| \leq Q\|u\|^2.
\ee
\end{lemma}
The second-order estimate \eqref{eq:projec-second-order1} is used as the local projection expansion in the convergence proof.

\subsection{\texorpdfstring{Network and stationarity}{Network and stationarity}}

{Let $G:=(\mathcal V,\mathcal E)$ denote the undirected communication network, where $\mathcal V:=\{1,2,\ldots,n\}$ is the node set and $\mathcal E$ is the edge set. We impose the following standard assumptions on $W$; see, e.g., \cite{zeng2018nonconvex,chen2021decentralized}.}

\begin{assumption} \label{assum:W}
We assume that the undirected graph $G$ is connected and $W$ satisfies the following conditions:
\begin{itemize}
  \item[(i)] {$W$ is symmetric and doubly stochastic, meaning that $W_{ij}\geq0$ for all $i,j\in[n]$ and every row and column sums to one, i.e., $W\mathbf{1}_n=W^\top\mathbf{1}_n=\mathbf{1}_n$. Moreover, $0<W_{ii}<1$ for all $i\in[n]$, and $W_{ij}=0$ whenever $i\ne j$ and $\{i,j\}\notin\mathcal E$;}
  \item[(ii)] \textit{Eigenvalues of $W$ lie in $(-1,1]$. The second largest singular value $\sigma_2$ of $W$ lies in $ [0,1)$.}
\end{itemize}
\end{assumption}
The above assumption implies $\bW \bx = \bx$ if and only if $x_1 = \dots = x_n$.
It also gives
\[
\|W-J\|_2=\sigma_2,\qquad
\|W-I_n\|_2=1-\lambda_{\min}(W)\le1+\sigma_2.
\]
The same bounds hold for the corresponding lifted matrices.

{Let $x_1,\ldots,x_n\in\Mcal$ be the nodes' local copies. Completing the square shows that
$\Pcal_\Mcal(\hat x)=\operatorname*{argmin}_{y\in\Mcal}\sum_{i=1}^n\|y-x_i\|^2$.
Any $\bar x\in\Pcal_\Mcal(\hat x)$ is therefore an induced arithmetic mean of $\{x_i\}_{i=1}^n$ on $\Mcal$ \cite{sarlette2009consensus}.}
{Using the global objective $F$ defined above, we define an $\epsilon$-stationary configuration for problem~\eqref{prob:original} as follows.}
\begin{definition} \label{def:station}
{The configuration $\{x_1,\ldots,x_n\}\subset\Mcal$ is called $\epsilon$-stationary for \eqref{prob:original} if there exists $\bar{x}\in\Pcal_{\Mcal}(\hat{x})$ such that $\frac{1}{n}\sum_{i=1}^n\|x_i-\bar{x}\|^2\le\epsilon$ and $\|\grad F(\bar{x})\|^2\le\epsilon$.}
\end{definition}
\subsection{\texorpdfstring{Projection and manifold-mean estimates}{Projection and manifold-mean estimates}}

{The convergence proof uses two uniform geometric estimates: Lipschitz variation of the tangent-space projector and a second-order estimate for the displacement between the Euclidean average and its manifold projection. Define}
\begin{equation}\label{eq:tangent-projector-lipschitz}
{
L_{\Pcal}:=
\sup_{\substack{x,y\in\Mcal\\x\ne y}}
\sup_{\substack{u\in\R^{d\times r}\\u\ne0}}
\frac{\|\Pcal_{T_x\Mcal}(u)-\Pcal_{T_y\Mcal}(u)\|}
{\|u\|\,\|x-y\|}.
}
\end{equation}
{The constant in \eqref{eq:tangent-projector-lipschitz} is finite because $\Mcal$ is compact and the tangent-space projector varies smoothly on $\Mcal$. The next estimate controls the distance between the Euclidean mean $\hat{x}$ and the manifold mean $\bar{x}$ by the squared consensus error.

Fix $\tau\in(0,\min\{R/2,1\})$ sufficiently small that $\Pcal_\Mcal$ is $C^2$ on the open tube}
\begin{equation}\label{eq:smooth-tube}
{
\mathcal U:=U_\Mcal(2\tau)
=\bigl\{z\in\R^{d\times r}:\operatorname{dist}(z,\Mcal)<2\tau\bigr\}.
}
\end{equation}
{We write $\mathcal U^n$ for its $n$-fold Cartesian product. Thus $\bar U_\Mcal(\tau)$ is a compact subset of $\mathcal U$. Define, once and for all,
\[
\begin{aligned}
M_2&:=\max\left\{1,\sup_{z\in\bar U_\Mcal(\tau)}
\|D^2\Pcal_\Mcal(z)\|_{\rm op}\right\}<\infty,\\
\|D^2\Pcal_\Mcal(z)\|_{\rm op}
&:=\sup_{\substack{u,v\in\R^{d\times r}\\\|u\|=\|v\|=1}}
\|D^2\Pcal_\Mcal(z)[u,v]\|.
\end{aligned}
\]
The tube $\mathcal U$ in \eqref{eq:smooth-tube}, radius $\tau$, and bound $M_2$ are used throughout the convergence analysis.}
\begin{lemma}[{\cite[Lemma 4.4]{deng2023decentralized}}]\label{lem:bound-con}
{Let $\mathbf x\in\Mcal^n$ have Euclidean mean $\hat x:=n^{-1}\sum_{i=1}^n x_i\in\bar U_\Mcal(\tau/2)$. Set $\bar x:=\Pcal_\Mcal(\hat x)$ and $\bar{\mathbf x}:=\mathbf1_n\otimes\bar x$. If $\max_{i\in[n]}\|x_i-\bar x\|\le\tau/2$, then}
\begin{equation} \label{eqn:0-1}
{
\|\bar{x}-\hat{x}\|\leq
M_2\frac{\|\mathbf{x}-\bar{\bx}\|^2}{n}.
}
\end{equation}
\end{lemma}
{These conditions make $\bar x$ unique and ensure $[\bar x,x_i]\subset\bar U_\Mcal(\tau/2)\subset\bar U_\Mcal(\tau)$ for every $i\in[n]$, so the bound $M_2$ applies on every segment used in the cited proof.
For the Stiefel manifold, \cite[Lemma 1]{chen2021local} gives the same type of mean estimate with coefficient $2\sqrt r$ whenever $\|\mathbf x-\bar{\mathbf x}\|^2\le n/2$. That coefficient is specific to the mean estimate and is not a value assigned to the uniform second-derivative bound $M_2$.}

\section{A Riemannian EXTRA for decentralized optimization}

{We develop REXTRA from the primal correction in the Euclidean EXTRA algorithm \cite{shi2015extra}. For $\Mcal=\R^{d\times r}$, the classical recursion is}
\be \label{eq:extra-E}
  \bx_{k+1}
  = \frac{I + \bW}{2}(2\bx_k - \bx_{k-1})
    - \alpha \big(\nabla f(\bx_k) - \nabla f(\bx_{k-1})\big),\qquad k\ge1,
\ee
{with the first step $\bx_1=\bW\bx_0-\alpha\nabla f(\bx_0)$. As shown in \cite[Section~6]{hong2017prox}, \eqref{eq:extra-E} admits an inexact augmented Lagrangian method (ALM) interpretation for the Euclidean consensus problem. Applying this ALM on $\Mcal^n$ yields a naive Riemannian extension. However, the manifold primal optimality conditions provide only tangential information, leaving the normal component uncontrolled. As tangent spaces vary, this component can affect subsequent tangential updates, so the Euclidean convergence guarantees do not carry over directly.}

{In this section, we develop Riemannian EXTRA (REXTRA) methods that incorporate a normal correction. We first present a projection-based method and then introduce its retraction-based counterpart.}

\subsection{Projection-based Riemannian EXTRA}

{We next introduce a primal-perspective EXTRA family with an auxiliary sequence $\{\bs_k\}$:}
\be \label{eq:extra-eqiu}
\begin{aligned}
    \bx_{k+1} & = \bW \bx_k + \bs_k, \\
    \bs_{k+1} & = (\bW - \bV)\bx_k + \bs_k - \alpha \big(\nabla f(\bx_{k+1})- \nabla f(\bx_k)\big),
\end{aligned}
\ee
{where $\alpha>0$ is the stepsize and $\bV:=V\otimes I_d$, with $V:=\theta I_n+(1-\theta)W$ and $\theta\in(0,\frac12]$. Eliminating $\bs_k$ gives $\bx_{k+2}=(I+\bW)\bx_{k+1}-\bV\bx_k-\alpha(\nabla f(\bx_{k+1})-\nabla f(\bx_k))$. With $\bs_0=-\alpha\nabla f(\bx_0)$, its first step is $\bx_1=\bW\bx_0-\alpha\nabla f(\bx_0)$. Hence, when $\theta=\frac12$, \eqref{eq:extra-eqiu} is exactly equivalent to \eqref{eq:extra-E}, including the specified first step; other values of $\theta$ give the generalized EXTRA family analyzed below.}

{In this primal-perspective formulation, the sequence $\{\bs_k\}$ is an auxiliary correction that acts as an approximate negative gradient direction for the global objective. For the Euclidean recursion, we denote $\hat{g}_k^{\rm E}:=\frac{1}{n}\sum_{i=1}^n\nabla f_i(x_{i,k})$ and $\hbg_k^{\rm E}:=\mathbf1_n\otimes\hat{g}_k^{\rm E}\in\R^{(nd)\times r}$, so that $\hbg_k^{\rm E}$ stacks $n$ copies of the averaged Euclidean gradient $\hat{g}_k^{\rm E}$. For the stated initialization, \eqref{eq:extra-eqiu} gives $\mathbf J\bs_k=-\alpha\mathbf J\nabla f(\bx_k)$, and $\|\bs_k+\alpha\hbg_k^{\rm E}\|$ measures the discrepancy from the averaged negative gradient direction. On the manifold, we control the corresponding tangent and normal components at the manifold mean and use them to derive a descent inequality for the global objective.}

{To adapt this primal representation to the manifold setting, we combine Riemannian gradients and metric projection with a normal correction in the network updates. For the projection-based method analyzed below, let $\bg_k:=\grad f(\bx_k)$ and $\by_k:=\bs_k+\alpha\bg_k$, where $\by_k$ is a shifted auxiliary state. The term $(I-\Pcal_{T_{\bx_k}\Mcal^n})(\by_k)$ feeds its local normal component into both network corrections. We use the update
\be \label{eq:R-extra}
\begin{aligned}
    \mathbf z_k &:= \bx_k+(I-\Pcal_{T_{\bx_k}\Mcal^n})(\by_k),\\
    \bx_{k+1} &= \Pcal_{\Mcal^n}\left(\bW\mathbf z_k+\Pcal_{T_{\bx_k}\Mcal^n}(\bs_k)\right),\\
    \bs_{k+1} &= \bs_k+(\bW-\bV)\mathbf z_k-\alpha(\bg_{k+1}-\bg_k).
\end{aligned}
\ee
Here, $\Pcal_{\Mcal^n}$ denotes the blockwise metric projection onto $\Mcal^n$, and $\grad f$ is the stacked Riemannian gradient. We refer to \eqref{eq:R-extra} as the \emph{projection-based Riemannian EXTRA} (projection-based REXTRA) iteration. The display is the compact mathematical form used in the analysis, while Algorithm~\ref{alg} gives the same update as pseudocode. When $\Mcal=\mathbb R^{d\times r}$, the normal-projection terms vanish and \eqref{eq:R-extra} reduces to the Euclidean primal-perspective EXTRA recursion \eqref{eq:extra-eqiu}; for $\theta=\tfrac12$ and the stated initialization, it is the classical Euclidean EXTRA recursion.}

\begin{algorithm}[H]
\caption{Projection-based Riemannian EXTRA for solving \eqref{prob:original}}
\begin{algorithmic}[1]
\REQUIRE Initial point $\bx_0\in \Mcal^n$, stepsize $\alpha>0$, set $\bs_0=-\alpha\grad f(\bx_0)$ and $k=0$.
\WHILE{the stopping condition is not met}
\STATE Set $\mathbf z_k:=\bx_k+(I-\Pcal_{T_{\bx_k}\Mcal^n})(\bs_k+\alpha\grad f(\bx_k))$.
\STATE $\bx_{k+1}=\Pcal_{\Mcal^n}\left(\bW\mathbf z_k+\Pcal_{T_{\bx_k}\Mcal^n}(\bs_k)\right)$.
\STATE $\bs_{k+1}=\bs_k+(\bW-\bV)\mathbf z_k-\alpha\big(\grad f(\bx_{k+1})-\grad f(\bx_k)\big)$.
\STATE Set $k\gets k+1$.
\ENDWHILE
\end{algorithmic}
\label{alg}
\end{algorithm}

{The normal correction distinguishes \eqref{eq:R-extra} from a naive Riemannian extension. Omitting it, by setting $\mathbf z_k=\mathbf x_k$ in \eqref{eq:R-extra}, gives the shifted auxiliary update $\mathbf y_{k+1}=\mathbf y_k+(\mathbf W-\mathbf V)\mathbf x_k$. This update does not feed the auxiliary normal component into the network mixing. As tangent spaces vary, the Euclidean ALM argument provides no control of the resulting normal-space error, so its convergence guarantee does not carry over. In \eqref{eq:R-extra}, the term $(I-\Pcal_{T_{\mathbf x_k}\Mcal^n})(\mathbf y_k)$ is instead included in the communicated variable and fed back through the auxiliary update. This feedback supplies the normal-error contraction used in the coupled analysis in Section~\ref{sec-con}.}

{Compared with decentralized manifold optimization algorithms that exchange both an iterate and a tracking direction, such as DRGTA \cite{chen2021decentralized} and DPRGT \cite{deng2023decentralized}, REXTRA communicates one local-variable-sized matrix per round. Both occurrences of network mixing in \eqref{eq:R-extra} use the message $\mathbf z_k$ defined there and are obtained from the single aggregate $\mathbf W\mathbf z_k$, because $(\mathbf W-\mathbf V)\mathbf z_k=\theta(\mathbf W\mathbf z_k-\mathbf z_k)$. Each node therefore sends only its block of $\mathbf z_k$ and computes its Riemannian gradient locally.}
Moreover, DRGTA and DPRGT often require multiple rounds of communication over $W$ per iteration, while REXTRA only needs a single communication round.

\subsection{Retraction-based Riemannian EXTRA}

{Exact projections $\Pcal_{\Mcal^n}$ may be expensive or unavailable on some manifolds. We therefore use the same primal correction and the same single communication round as in \eqref{eq:R-extra}, but replace only its final feasibility operation by a retraction $\Rcal$ on $\Mcal$. Using the same $\bg_k$, $\by_k$, and $\mathbf z_k$ as in \eqref{eq:R-extra}, the update becomes}
\be \label{eq:retrac-R-extra}
{
\begin{aligned}
    \boldsymbol\xi_k&:=\Pcal_{T_{\bx_k}\Mcal^n}
    \left(\bW\mathbf z_k+\Pcal_{T_{\bx_k}\Mcal^n}(\bs_k)-\bx_k\right),\\
    \bx_{k+1}&:=\Rcal_{\bx_k}(\boldsymbol\xi_k),\\
    \bs_{k+1}&:=\bs_k+(\bW-\bV)\mathbf z_k
    -\alpha\big(\grad f(\bx_{k+1})-\grad f(\bx_k)\big).
\end{aligned}
}
\ee
{We call \eqref{eq:retrac-R-extra} the \emph{retraction-based Riemannian EXTRA} (retraction-based REXTRA) iteration. In particular, when $\Mcal$ is Euclidean, $\Rcal_x(\xi)=x+\xi$ and all tangent projections are identities, so \eqref{eq:retrac-R-extra} reduces exactly to the primal EXTRA recursion \eqref{eq:extra-eqiu}. The retraction recurrence differs from Algorithm~\ref{alg} only by replacing $\Pcal_{\Mcal^n}(\bW\mathbf z_k+\Pcal_{T_{\bx_k}\Mcal^n}(\bs_k))$ with $\Rcal_{\bx_k}(\boldsymbol\xi_k)$. This exact correspondence makes its second-order perturbation relative to Algorithm~\ref{alg} explicit and is used in the convergence proof in Section~\ref{sec-con}.}

\begin{algorithm}[H]
\caption{Retraction-based Riemannian EXTRA for solving \eqref{prob:original}}
\begin{algorithmic}[1]
\REQUIRE {Initial point $\bx_0\in\Mcal^n$, stepsize $\alpha>0$, set $\bs_0=-\alpha\grad f(\bx_0)$ and $k=0$.}
\WHILE{the stopping condition is not met}
\STATE {Set $\mathbf z_k:=\bx_k+(I-\Pcal_{T_{\bx_k}\Mcal^n})(\bs_k+\alpha\grad f(\bx_k))$.}
\STATE {$\bx_{k+1}:=\Rcal_{\bx_k}\!\left(\Pcal_{T_{\bx_k}\Mcal^n}(\bW\mathbf z_k+\Pcal_{T_{\bx_k}\Mcal^n}(\bs_k)-\bx_k)\right)$.}
\STATE {$\bs_{k+1}:=\bs_k+(\bW-\bV)\mathbf z_k-\alpha\big(\grad f(\bx_{k+1})-\grad f(\bx_k)\big)$.}
\STATE {Set $k\gets k+1$.}
\ENDWHILE
\end{algorithmic}
\label{alg-retrac}
\end{algorithm}

{Algorithm~\ref{alg-retrac} executes the retraction recurrence \eqref{eq:retrac-R-extra} with the same initialization as Algorithm~\ref{alg}. It retains the same normal correction in both the communicated variable and the auxiliary update; only the final feasibility map changes to a retraction. The only mixed quantity is $\mathbf z_k$, while $(\bW-\bV)\mathbf z_k=\theta(\bW\mathbf z_k-\mathbf z_k)$ is obtained from the same received aggregate. Thus the algorithm uses one local-variable-sized communication round per iteration. Lemma~\ref{lem:retr-local-transfer} below controls the second-order retraction--projection discrepancy, and Theorem~\ref{thm:retr-global-conv} establishes the stationarity--consensus guarantee and rate under its local initialization conditions.}

\section{Convergence analysis}\label{sec-con}
We analyze the projection-based and retraction-based REXTRA updates separately. In both cases, the proof couples the consensus error with the tangent and normal components of the auxiliary correction. The resulting error recursion and a descent estimate for the manifold mean give a joint $\mathcal O(1/K)$ best-iterate bound.

\begin{assumption}\label{ass:lip}
For every $i\in[n]$, the Euclidean gradient of $f_i$ is $L_f$-Lipschitz continuous on ${\rm conv}(\Mcal)$, and $\|\nabla f_i(x)\|\le L_g$ for every $x\in\Mcal$.
\end{assumption}
Throughout this section, Assumptions~\ref{assum:W} and~\ref{ass:lip} hold, and $0<\theta\le1/2$.
Together with the geometric estimates in Section~2, Assumption~\ref{ass:lip} gives a constant
$L:=\max\{L_f+L_g/R,L_g,L_f+L_gL_{\Pcal}\}$ such that, for all $x,y\in\Mcal$,
\begin{equation}\label{eq:quad}
\begin{aligned}
f_i(y)&\le f_i(x)+\iprod{\grad f_i(x)}{y-x}+\frac L2\|y-x\|^2,\\
\|\grad f_i(y)-\grad f_i(x)\|&\le L\|y-x\|;
\end{aligned}
\end{equation}
see \cite{deng2023decentralized}. The first inequality also holds for $F$. Smoothness on ${\rm conv}(\Mcal)$ supplies the ambient estimates along the chords joining manifold points. The analysis equally applies under any smoothness assumption that ensures \eqref{eq:quad} uniformly.

Retain the tube $\mathcal U$ defined in \eqref{eq:smooth-tube}, together with the fixed radius $\tau$ and bound $M_2$ above, and set
\begin{equation}\label{eq:gradient-bound-constants}
D:=\sqrt n L_g,\qquad B:=1+L_g.
\end{equation}
Thus $\|\grad F(x)\|\le L_g<B$ and $\|\grad f(\mathbf x)\|\le D$ on $\Mcal$ and $\Mcal^n$, respectively. By \eqref{eq:lip-proj}, the projection is $2$-Lipschitz on $\bar U_\Mcal(\tau)\subset\mathcal U$.

\subsection{Projection-based REXTRA}\label{sec-projection}
We first construct a norm for the linear disagreement dynamics. Set
\begin{equation}\label{eq:local-P-def-split}
\mathbf P:=\begin{pmatrix}
\mathbf W-\mathbf J&I\\
\mathbf W-\mathbf V&I-\mathbf J
\end{pmatrix},\qquad
\mathbf N:=I+\theta(\mathbf W-I).
\end{equation}

\begin{lemma}\label{lem:A}
If $\rho(A)<1$, there exist a symmetric matrix $H\succeq I$ and $\nu\in(0,1)$ such that
\[
\|AZ\|_H\le\nu\|Z\|_H,\qquad
\|Z\|_H^2:=\operatorname{tr}(Z^\top HZ)
\]
for every matrix $Z$ of compatible size.
\end{lemma}
\begin{proof}
Choose $\nu\in(\rho(A),1)$ and define
\[
H:=\sum_{j=0}^{\infty}\bigl[(A/\nu)^j\bigr]^\top(A/\nu)^j.
\]
The series converges because $\rho(A/\nu)<1$. Its $j=0$ term is $I$, and every term with $j\ge1$ is positive semidefinite. Hence $H\succeq I$, which also gives $\|Z\|_F\le\|Z\|_H$. Moreover, $A^\top HA=\nu^2(H-I)\preceq\nu^2H$. Taking traces against $ZZ^\top$ proves the contraction inequality.
\end{proof}

\begin{lemma}\label{lem:local-P-stable-split}
Under Assumption~\ref{assum:W} and $0<\theta\le1/2$, $\rho(\mathbf P)<1$. Write $1=\lambda_1(W)>\lambda_2(W)\ge\cdots\ge\lambda_n(W)>-1$ for the ordered eigenvalues of $W$, and define
\[
\sigma:=\max_{2\le i\le n}|1-\theta+\theta\lambda_i(W)|.
\]
Then $\sigma\in(0,1)$ and $\|\mathbf N\mathbf v\|\le\sigma\|\mathbf v\|$ whenever $\mathbf J\mathbf v=0$.
\end{lemma}
\begin{proof}
Diagonalize the symmetric matrix $W$. Its consensus eigenvalue gives the block $\left(\begin{smallmatrix}0&1\\0&0\end{smallmatrix}\right)$ of $\mathbf P$. For $i=2,\ldots,n$, put $\omega:=\lambda_i(W)\in(-1,1)$. The corresponding block and characteristic polynomial are
\[
P_\omega=\begin{pmatrix}\omega&1\\\theta(\omega-1)&1\end{pmatrix},\qquad
p_\omega(\lambda)=\lambda^2-(1+\omega)\lambda+\theta+(1-\theta)\omega.
\]
Writing $p_\omega(\lambda)=\lambda^2+a_1\lambda+a_0$, the Jury inequalities are
\[
\begin{aligned}
1+a_1+a_0&=\theta(1-\omega)>0,\\
1-a_1+a_0&=2+\theta+(2-\theta)\omega>0,\\
1-a_0&=(1-\theta)(1-\omega)>0.
\end{aligned}
\]
Thus all roots lie in the open unit disk. On the disagreement subspace, the eigenvalues of $\mathbf N$ are $1-\theta+\theta\lambda_i(W)$, $i=2,\ldots,n$. They belong to $(0,1)$, so their largest absolute value is the stated $\sigma\in(0,1)$. Symmetry of $\mathbf N$ gives the claimed norm bound.
\end{proof}

Set $\nu_P:=(1+\rho(\mathbf P))/2$ and choose $H$ by the series in the proof of Lemma~\ref{lem:A}, with $A=\mathbf P$ and $\nu=\nu_P$. In particular, $H\succeq I$. Fix
\begin{equation}\label{eq:fixed-network-constants}
h_H:=\sqrt{\lambda_{\max}(H)},\qquad
\beta:=\theta(1+\sigma_2),\qquad
\nu:=\max\{\nu_P,\sigma\}<1.
\end{equation}
Since $\mathbf W-\mathbf V=\theta(\mathbf W-I)$, Assumption~\ref{assum:W} yields $\|\mathbf W-\mathbf V\|_2\le\beta$.
All subsequent matrix $H$-norms are columnwise norms as defined in that lemma. At iteration $k$, define
\begin{equation}\label{eq:analysis-notation}
\begin{aligned}
\mathbf g_k&:=\grad f(\mathbf x_k),&
\hat g_k&:=\frac1n\sum_{i=1}^n\grad f_i(x_{i,k}),&
\hbg_k&:=\mathbf1_n\otimes\hat g_k,\\
\mathbf d_k&:=\mathbf g_k-\hbg_k,&
\mathbf y_k&:=\mathbf s_k+\alpha\mathbf g_k,&
\mathbf e_k&:=\mathbf x_k-\bar{\mathbf x}_k.
\end{aligned}
\end{equation}
Thus $\mathbf J\mathbf d_k=0$. Introduce the linear projectors
\[
T_k:=\Pcal_{T_{\bar{\mathbf x}_k}\Mcal^n},\qquad
Q_k:=\Pcal_{T_{\mathbf x_k}\Mcal^n}.
\]
Because the center is replicated, $T_k$ commutes with every matrix acting only on the node index, including $\mathbf W$, $\mathbf V$, and $\mathbf J$. Define
\begin{equation}\label{eq:local-errors-def-split}
\mathbf n_k:=(I-T_k)(\mathbf y_k),\qquad
\mathbf r_k:=T_k\bigl(\mathbf s_k+\alpha\hbg_k\bigr)
=T_k\bigl(\mathbf y_k-\alpha\mathbf d_k\bigr),
\end{equation}
and
\begin{equation}\label{eq:local-Z-def-split}
\mathcal Z_k:=\left\|\begin{pmatrix}\mathbf e_k\\\mathbf r_k\end{pmatrix}\right\|_H+\|\mathbf n_k\|.
\end{equation}
The variables $\mathbf n_k$ and $\mathbf r_k$ measure the normal auxiliary component and the tangent discrepancy from the averaged negative-gradient direction. The choice $H\succeq I$ gives the explicit comparisons
\begin{equation}\label{eq:explicit-norm-comparison}
\|\mathbf e_k\|,\ \|\mathbf r_k\|\le\mathcal Z_k,\qquad
\|\mathbf e_k\|+\|\mathbf r_k\|+\|\mathbf n_k\|\le\sqrt2\mathcal Z_k,\qquad
\left\|\begin{pmatrix}\mathbf v\\\mathbf w\end{pmatrix}\right\|_H
\le h_H(\|\mathbf v\|+\|\mathbf w\|).
\end{equation}

\begin{lemma}\label{lem:local-secant-normal-split}
For $\mathbf x\in\Mcal^n$ and $\bar{\mathbf x}=\mathbf1_n\otimes\bar x$, $\bar x\in\Mcal$,
\begin{equation}\label{eq:local-secant-normal-general-split}
\|(I-\Pcal_{T_{\bar{\mathbf x}}\Mcal^n})(\mathbf x-\bar{\mathbf x})\|
\le\frac1{2R}\|\mathbf x-\bar{\mathbf x}\|^2.
\end{equation}
\end{lemma}
\begin{proof}
For each block, put $e_i=x_i-\bar x$ and $v_i=(I-\Pcal_{T_{\bar x}\Mcal})(e_i)$. Proximal smoothness, \eqref{proximally-0}, gives
\[
\|v_i\|^2=\iprod{v_i}{e_i}\le\frac{\|v_i\|}{2R}\|e_i\|^2.
\]
Hence $\|v_i\|\le\|e_i\|^2/(2R)$. Square and sum, using $\sum_i\|e_i\|^4\le(\sum_i\|e_i\|^2)^2$.
\end{proof}

The next lemma contains the one-step estimates needed for both stability and objective descent. In particular, the center $\bar x_{k+1}$ is constructed before its displacement from $\bar x_k$ is used in the auxiliary-error recursions.

\begin{lemma}\label{lem:local-recursion-split}
Let $\mathcal U$ be the tube defined in \eqref{eq:smooth-tube}. Suppose that, at iteration $k$, $\mathbf x_k\in\Mcal^n$, $\hat x_k\in\mathcal U$, and $\mathbf J\mathbf y_k=0$. Define the input of the product metric projection by
\[
\mathbf p_k:=\mathbf W\mathbf z_k+
\Pcal_{T_{\mathbf x_k}\Mcal^n}(\mathbf s_k).
\]
The constants $\delta_p,A_p,D_p,m_p,j_p>0$ defined in the proof are independent of $k$ and $\alpha$. If $\alpha>0$ and $\alpha+\mathcal Z_k\le\delta_p$, the projection-based update \eqref{eq:R-extra} from iteration $k$ to iteration $k+1$ is well defined, with $\mathbf x_{k+1}=\Pcal_{\Mcal^n}(\mathbf p_k)$, $\mathbf p_k\in\mathcal U^n$, and $\hat x_{k+1}\in\mathcal U$. With $G_k:=\|\grad F(\bar x_k)\|$ and $E_k:=(\alpha+\mathcal Z_k)\mathcal Z_k$, the update satisfies
\begin{align}
\mathcal Z_{k+1}
&\le\nu\mathcal Z_k+A_p E_k+D_p\alpha^2G_k,
\label{eq:local-Z-recur-split}\\
\|\bar x_{k+1}-\bar x_k+\alpha\grad F(\bar x_k)\|
&\le m_pE_k+j_p\alpha^2G_k^2.
\label{eq:local-center-displacement-remainder-split}
\end{align}
\end{lemma}

\begin{proof}
Put $\ell_k:=\mathcal Z_k+\alpha G_k$. The following fixed coefficients are used in the estimates at iteration $k$:
\begin{equation}\label{eq:projection-input-constants}
\begin{aligned}
c_y&:=\max\{\sqrt2,D\},&
c_\eta&:=L_{\Pcal}c_y,\\
c_\gamma&:=\beta c_\eta,&
c_b&:=\max\{D,\sqrt2(1+\beta)+c_\gamma\},\\
c_\Delta&:=(1+\sigma_2)c_\eta+L_{\Pcal}D,&c_p&:=\sqrt2+c_\Delta,\\
c_h&:=L/\sqrt n,&
c_i&:=\max\{\sqrt n,L+1+\sigma_2+\sqrt2+c_\Delta\}.
\end{aligned}
\end{equation}
For the projection and center expansions, fix
\begin{equation}\label{eq:projection-expansion-constants}
\begin{aligned}
c_x&:=\sqrt2 M_2D+2c_\Delta+\frac{M_2}{2}c_p^2+\frac1{2R},
&c_\zeta&:=\sqrt2+c_x,\\
c_c&:=\frac{M_2}{\sqrt n}+c_x+\frac{M_2c_\zeta^2}{2\sqrt n},
&c_e&:=c_x+c_c+\frac{M_2}{\sqrt n}.
\end{aligned}
\end{equation}
These formulas involve only the fixed bounds above. Choose the radius
\begin{equation}\label{eq:projection-radius-explicit}
\delta_p:=\min\left\{1,\frac\tau2,
\frac{\tau}{2\max\{1,D+c_p\}},
\frac{\tau\sqrt n}{2c_\zeta},
\frac{R}{2\max\{1,L_g\}}\right\}.
\end{equation}
In particular, $E_k\le\mathcal Z_k$ whenever $\alpha+\mathcal Z_k\le\delta_p$.

\emph{Estimates at iteration $k$.}
The identities in \eqref{eq:local-errors-def-split} give
\begin{equation}\label{eq:local-y-decomposition}
\mathbf y_k=\mathbf r_k+\mathbf n_k+\alpha T_k(\mathbf d_k),\qquad
\|\mathbf y_k\|\le\sqrt2\mathcal Z_k+\alpha D
\le c_y(\alpha+\mathcal Z_k),
\end{equation}
where $\|\mathbf d_k\|=\|(I-\mathbf J)\mathbf g_k\|\le D$. Since $T_k$ commutes with $\mathbf J$, we also have $\mathbf J\mathbf r_k=\mathbf J\mathbf n_k=0$. The bounds $\operatorname{dist}(\hat x_k,\Mcal)\le\|\mathbf e_k\|/\sqrt n$ and $\max_i\|x_{i,k}-\bar x_k\|\le\|\mathbf e_k\|\le\delta_p\le\tau/2$ verify the domain conditions of Lemma~\ref{lem:bound-con}. Together with Lemma~\ref{lem:local-secant-normal-split}, it yields
\begin{equation}\label{eq:local-current-bounds}
\begin{aligned}
\|(I-T_k)(\mathbf e_k)\|&\le\frac{\mathcal Z_k^2}{2R},&
\|\mathbf J\mathbf e_k\|&\le\frac{M_2}{\sqrt n}\mathcal Z_k^2,\\
\|T_k-Q_k\|_{\rm op}&\le L_{\Pcal}\mathcal Z_k,&
\|(I-T_k)(\mathbf g_k)\|&\le L_{\Pcal}D\mathcal Z_k.
\end{aligned}
\end{equation}
The last estimate uses $Q_k(\mathbf g_k)=\mathbf g_k$. Define $\boldsymbol\eta_k:=(T_k-Q_k)(\mathbf y_k)$ and $\boldsymbol\gamma_k:=(\mathbf W-\mathbf V)\boldsymbol\eta_k$. By \eqref{eq:local-y-decomposition}--\eqref{eq:local-current-bounds}, $\|\boldsymbol\eta_k\|\le\|T_k-Q_k\|_{\rm op}\,\|\mathbf y_k\|\le L_{\Pcal}c_y\mathcal Z_k(\alpha+\mathcal Z_k)=c_\eta E_k$. The bound $\|\mathbf W-\mathbf V\|_2\le\beta$ then gives $\|\boldsymbol\gamma_k\|\le\|\mathbf W-\mathbf V\|_2\,\|\boldsymbol\eta_k\|\le\beta c_\eta E_k=c_\gamma E_k$, with $c_\eta$ and $c_\gamma$ fixed in \eqref{eq:projection-input-constants}. The identity $I-Q_k=(I-T_k)+(T_k-Q_k)$ gives $\mathbf z_k=\mathbf x_k+(I-Q_k)\mathbf y_k=\bar{\mathbf x}_k+\mathbf e_k+\mathbf n_k+\boldsymbol\eta_k$, while \eqref{eq:R-extra} gives $\mathbf y_{k+1}=\mathbf s_{k+1}+\alpha\mathbf g_{k+1}=\mathbf y_k+(\mathbf W-\mathbf V)\mathbf z_k$. Since $(\mathbf W-\mathbf V)\bar{\mathbf x}_k=0$, this yields the exact shifted update
\begin{equation}\label{eq:local-y-update}
\mathbf y_{k+1}=\mathbf y_k+(\mathbf W-\mathbf V)(\mathbf e_k+\mathbf n_k)+\boldsymbol\gamma_k.
\end{equation}
Since $\mathbf J(\mathbf W-\mathbf V)=0$, the shifted recursion \eqref{eq:local-y-update} preserves $\mathbf J\mathbf y_k=0$. Also, \eqref{eq:explicit-norm-comparison} gives
\[
\|\mathbf y_{k+1}\|
\le D\alpha+\bigl(\sqrt2(1+\beta)+c_\gamma\bigr)\mathcal Z_k
\le c_b(\alpha+\mathcal Z_k).
\]

\emph{Projection update and center at iteration $k+1$.}
Let
\[
h_k:=\Pcal_{T_{\bar x_k}\Mcal}(\hat g_k),\qquad
\mathbf h_k:=\mathbf1_n\otimes h_k=T_k(\hbg_k).
\]
We have $\|h_k\|\le L_g$ and, by \eqref{eq:quad},
\begin{equation}\label{eq:local-center-direction-error-split}
\|h_k-\grad F(\bar x_k)\|\le\frac L{\sqrt n}\|\mathbf e_k\|
\le c_h\mathcal Z_k.
\end{equation}
Starting from the projection input in \eqref{eq:R-extra}, substitute $\mathbf s_k=\mathbf y_k-\alpha\mathbf g_k$ from \eqref{eq:analysis-notation} and use \eqref{eq:local-errors-def-split} and \eqref{eq:local-y-decomposition} to obtain
\begin{equation}\label{eq:local-x-argument-split}
\begin{aligned}
\mathbf p_k
&=\mathbf W\mathbf z_k+Q_k(\mathbf s_k)\\
&=\mathbf W(\bar{\mathbf x}_k+\mathbf e_k+\mathbf n_k+\boldsymbol\eta_k)
+Q_k(\mathbf y_k-\alpha\mathbf g_k)\\
&=\bar{\mathbf x}_k+\mathbf W\mathbf e_k+\mathbf W\mathbf n_k
+T_k(\mathbf y_k)+(\mathbf W-I)\boldsymbol\eta_k-\alpha\mathbf g_k\\
&=\bar{\mathbf x}_k-\alpha\mathbf h_k+\mathbf W\mathbf e_k
+\mathbf r_k+\mathbf W\mathbf n_k
+(\mathbf W-I)\boldsymbol\eta_k-\alpha(I-T_k)(\mathbf g_k)\\
&=\mathbf a_k+\mathbf W\mathbf e_k+\mathbf r_k
+\mathbf W\mathbf n_k+\boldsymbol\Delta_k,
\end{aligned}
\end{equation}
where $\mathbf a_k:=\bar{\mathbf x}_k-\alpha\mathbf h_k$ and $\boldsymbol\Delta_k:=(\mathbf W-I)\boldsymbol\eta_k-\alpha(I-T_k)(\mathbf g_k)$. Here we used $\mathbf W\bar{\mathbf x}_k=\bar{\mathbf x}_k$, $Q_k(\mathbf g_k)=\mathbf g_k$, $Q_k(\mathbf y_k)=T_k(\mathbf y_k)-\boldsymbol\eta_k$, and $T_k(\mathbf d_k)=T_k(\mathbf g_k)-\mathbf h_k$. Assumption~\ref{assum:W}, \eqref{eq:local-current-bounds}, and the estimate for $\boldsymbol\eta_k$ above give $\|\boldsymbol\Delta_k\|\le\|\mathbf W-I\|_2\|\boldsymbol\eta_k\|+\alpha\|(I-T_k)(\mathbf g_k)\|\le(1+\sigma_2)c_\eta E_k+\alpha L_{\Pcal}D\mathcal Z_k\le c_\Delta E_k$, since $\alpha\mathcal Z_k\le E_k$ and $c_\Delta$ is fixed in \eqref{eq:projection-input-constants}.
Since $\|\mathbf W\|_2=1$, \eqref{eq:explicit-norm-comparison} implies
\[
\|\mathbf a_k-\bar{\mathbf x}_k\|\le D\alpha,\quad
\|\mathbf p_k-\mathbf a_k\|\le c_p\mathcal Z_k,\quad
\|\mathbf p_k-\mathbf x_k\|\le c_i\ell_k.
\]
For the last bound, use $\|\mathbf h_k\|\le\sqrt nG_k+L\mathcal Z_k$ and $\alpha\le1$. Every point on $[\mathbf a_k,\mathbf p_k]$ is within distance
$D\alpha+c_p\mathcal Z_k\le\tau/2$ of $\Mcal^n$, by \eqref{eq:projection-radius-explicit}. Hence both
\[
\mathbf x_{k+1}=\Pcal_{\Mcal^n}(\mathbf p_k),\qquad
\boldsymbol\zeta_k:=\Pcal_{\Mcal^n}(\mathbf a_k)=\mathbf1_n\otimes z_k,
\quad z_k:=\Pcal_\Mcal(\bar x_k-\alpha h_k)
\]
are defined. The blockwise projection inherits the second-derivative bound $M_2$ on this segment. Taylor's theorem at $\mathbf a_k$ therefore gives $\mathbf x_{k+1}-\boldsymbol\zeta_k=D\Pcal_{\Mcal^n}(\mathbf a_k)[\mathbf p_k-\mathbf a_k]+\mathbf R_k$, where $\|\mathbf R_k\|\le\frac{M_2}{2}\|\mathbf p_k-\mathbf a_k\|^2\le\frac{M_2c_p^2}{2}\mathcal Z_k^2$. To replace this derivative by $T_k$, note that $D\Pcal_{\Mcal^n}(\bar{\mathbf x}_k)=T_k$. Every block along $[\bar{\mathbf x}_k,\mathbf a_k]$ also lies in $\bar U_\Mcal(\tau/2)$, so $\|D\Pcal_{\Mcal^n}(\mathbf a_k)-T_k\|_{\rm op}\le M_2\|\mathbf a_k-\bar{\mathbf x}_k\|\le M_2D\alpha$. Moreover, \eqref{eq:lip-proj} gives $\|D\Pcal_{\Mcal^n}(\mathbf a_k)\|_{\rm op}\le2$. Substituting \eqref{eq:local-x-argument-split} into the Taylor formula yields
\[
\begin{aligned}
\mathbf x_{k+1}-\boldsymbol\zeta_k
&=T_k(\mathbf W\mathbf e_k+\mathbf r_k+\mathbf W\mathbf n_k)+\boldsymbol\epsilon_k,\\
\boldsymbol\epsilon_k
&:=\bigl[D\Pcal_{\Mcal^n}(\mathbf a_k)-T_k\bigr]
[\mathbf W\mathbf e_k+\mathbf r_k+\mathbf W\mathbf n_k]\\
&\quad+D\Pcal_{\Mcal^n}(\mathbf a_k)[\boldsymbol\Delta_k]+\mathbf R_k.
\end{aligned}
\]
By \eqref{eq:explicit-norm-comparison}, $\|\mathbf W\mathbf e_k+\mathbf r_k+\mathbf W\mathbf n_k\|\le\sqrt2\mathcal Z_k$. Combining this with \eqref{eq:local-x-argument-split}, the derivative difference, the $\boldsymbol\Delta_k$ term, and the Taylor remainder give, respectively,
\[
\|\boldsymbol\epsilon_k\|
\le\sqrt2 M_2D\alpha\mathcal Z_k+2c_\Delta E_k
+\frac{M_2c_p^2}{2}\mathcal Z_k^2.
\]
Finally, $T_k$ commutes with $\mathbf W$, $T_k(\mathbf n_k)=0$, and $T_k(\mathbf r_k)=\mathbf r_k$. Hence $T_k(\mathbf W\mathbf e_k)=\mathbf W\mathbf e_k-\mathbf W(I-T_k)(\mathbf e_k)$. Define $\boldsymbol\epsilon_{x,k}:=\boldsymbol\epsilon_k-\mathbf W(I-T_k)(\mathbf e_k)$. Equation~\eqref{eq:local-current-bounds} gives $\|\mathbf W(I-T_k)(\mathbf e_k)\|\le\mathcal Z_k^2/(2R)$. Adding this term to the preceding bound and using $\alpha\mathcal Z_k\le E_k$ and $\mathcal Z_k^2\le E_k$, we obtain
\begin{equation}\label{eq:local-x-zeta-simple-split}
\mathbf x_{k+1}-\boldsymbol\zeta_k
=\mathbf W\mathbf e_k+\mathbf r_k+\boldsymbol\epsilon_{x,k},\qquad
\|\boldsymbol\epsilon_{x,k}\|\le c_x E_k,
\end{equation}
where $c_x$ is fixed in \eqref{eq:projection-expansion-constants}. Since $E_k\le\mathcal Z_k$, \eqref{eq:explicit-norm-comparison} and \eqref{eq:local-x-zeta-simple-split} give $\|\mathbf x_{k+1}-\boldsymbol\zeta_k\|\le(\sqrt2+c_x)\mathcal Z_k=c_\zeta\mathcal Z_k$. The average of every point on $[\boldsymbol\zeta_k,\mathbf x_{k+1}]$ is within $c_\zeta\mathcal Z_k/\sqrt n\le\tau/2$ of $z_k\in\Mcal$. Thus the center $\bar x_{k+1}$ is defined before it is used below.

On this segment, consider the center map
\[
\mathcal C(\mathbf v):=\mathbf1_n\otimes\Pcal_\Mcal\left(\frac1n\sum_i v_i\right).
\]
The averaging map has norm $1/\sqrt n$, while replication by $\mathbf1_n\otimes(\cdot)$ has norm $\sqrt n$. Thus the chain rule and the fixed second-derivative bound for $\Pcal_\Mcal$ give $\|D^2\mathcal C(\mathbf v)\|_{\rm op}\le\sqrt n M_2(1/\sqrt n)^2=M_2/\sqrt n$ along the segment. Since $\mathcal C(\boldsymbol\zeta_k)=\boldsymbol\zeta_k$ and $\mathcal C(\mathbf x_{k+1})=\bar{\mathbf x}_{k+1}$, Taylor's theorem at $\boldsymbol\zeta_k$ gives $\bar{\mathbf x}_{k+1}-\boldsymbol\zeta_k=D\mathcal C(\boldsymbol\zeta_k)[\mathbf x_{k+1}-\boldsymbol\zeta_k]+\mathbf R_{c,k}$, where $\|\mathbf R_{c,k}\|\le M_2\|\mathbf x_{k+1}-\boldsymbol\zeta_k\|^2/(2\sqrt n)$. Because $z_k\in\Mcal$, $D\Pcal_\Mcal(z_k)=\Pcal_{T_{z_k}\Mcal}$, so $D\mathcal C(\boldsymbol\zeta_k)=\mathbf J\Pcal_{T_{\boldsymbol\zeta_k}\Mcal^n}$. This tangent projector has the same block at every node and therefore commutes with $\mathbf W$ and $\mathbf J$. Substituting \eqref{eq:local-x-zeta-simple-split} and using $\mathbf J\mathbf W=\mathbf J$ and $\mathbf J\mathbf r_k=0$, the linear term reduces to $\Pcal_{T_{\boldsymbol\zeta_k}\Mcal^n}(\mathbf J\mathbf e_k+\mathbf J\boldsymbol\epsilon_{x,k})$. In particular, the $\mathbf r_k$ term vanishes, and $\mathbf J\mathbf e_k$ is bounded by the squared disagreement in \eqref{eq:local-current-bounds}. Since both projectors have norm at most one, \eqref{eq:local-current-bounds}, \eqref{eq:local-x-zeta-simple-split}, and the bound $\|\mathbf x_{k+1}-\boldsymbol\zeta_k\|\le c_\zeta\mathcal Z_k$ yield
\begin{equation}\label{eq:local-center-shift-split}
\begin{aligned}
\|\bar{\mathbf x}_{k+1}-\boldsymbol\zeta_k\|
&\le\|\mathbf J\mathbf e_k\|+\|\boldsymbol\epsilon_{x,k}\|
+\frac{M_2}{2\sqrt n}\|\mathbf x_{k+1}-\boldsymbol\zeta_k\|^2\\
&\le\frac{M_2}{\sqrt n}\mathcal Z_k^2+c_xE_k
+\frac{M_2c_\zeta^2}{2\sqrt n}\mathcal Z_k^2\\
&\le\left(\frac{M_2}{\sqrt n}+c_x+\frac{M_2c_\zeta^2}{2\sqrt n}\right)E_k
=c_cE_k,
\end{aligned}
\end{equation}
where the last inequality uses $\mathcal Z_k^2\le E_k$, and $c_c$ is the fixed coefficient in \eqref{eq:projection-expansion-constants}.
Subtracting the center from \eqref{eq:local-x-zeta-simple-split} yields
\begin{equation}\label{eq:local-e-recursion-split}
\mathbf e_{k+1}=(\mathbf W-\mathbf J)\mathbf e_k+\mathbf r_k+\boldsymbol\epsilon_{e,k},\qquad
\|\boldsymbol\epsilon_{e,k}\|\le c_e E_k,
\end{equation}
where $\boldsymbol\epsilon_{e,k}:=\boldsymbol\epsilon_{x,k}-(\bar{\mathbf x}_{k+1}-\boldsymbol\zeta_k)+\mathbf J\mathbf e_k$.

Recall that $z_k:=\Pcal_\Mcal(\bar x_k-\alpha h_k)$ and $\boldsymbol\zeta_k=\mathbf1_n\otimes z_k$. Since $h_k\in T_{\bar x_k}\Mcal$ and $\alpha\|h_k\|\le R/2$ by \eqref{eq:projection-radius-explicit}, Lemma~\ref{lemma-project} gives $z_k=\bar x_k-\alpha h_k+u_k$, where $\|u_k\|\le Q\alpha^2\|h_k\|^2$. Adding and subtracting $z_k$ and $\alpha h_k$, we obtain the exact decomposition
\[
\begin{aligned}
\bar x_{k+1}-\bar x_k+\alpha\grad F(\bar x_k)
&=(\bar x_{k+1}-z_k)+(z_k-\bar x_k+\alpha h_k)
+\alpha\bigl(\grad F(\bar x_k)-h_k\bigr)\\
&=(\bar x_{k+1}-z_k)+u_k
+\alpha\bigl(\grad F(\bar x_k)-h_k\bigr).
\end{aligned}
\]
The three terms are controlled, respectively, by the center estimate \eqref{eq:local-center-shift-split}, the projection remainder above, and the direction estimate \eqref{eq:local-center-direction-error-split}. In particular, $\|\bar x_{k+1}-z_k\|=\|\bar{\mathbf x}_{k+1}-\boldsymbol\zeta_k\|/\sqrt n$ and $\|h_k\|\le G_k+c_h\mathcal Z_k$. The triangle inequality and $(a+b)^2\le2a^2+2b^2$ therefore give
\[
\begin{aligned}
\|\bar x_{k+1}-\bar x_k+\alpha\grad F(\bar x_k)\|
&\le\frac1{\sqrt n}\|\bar{\mathbf x}_{k+1}-\boldsymbol\zeta_k\|
+Q\alpha^2\|h_k\|^2+\alpha\|h_k-\grad F(\bar x_k)\|\\
&\le\frac{c_c}{\sqrt n}E_k+c_h\alpha\mathcal Z_k
+2Q\alpha^2G_k^2+2Qc_h^2\alpha^2\mathcal Z_k^2.
\end{aligned}
\]
Since $\alpha\le1$, $\alpha\mathcal Z_k\le E_k$ and $\alpha^2\mathcal Z_k^2\le E_k$. This proves \eqref{eq:local-center-displacement-remainder-split} with the fixed choices
\begin{equation}\label{eq:projection-mean-constants}
m_p:=c_h+\frac{c_c}{\sqrt n}+2Qc_h^2,\qquad j_p:=2Q.
\end{equation}
The average $\hat x_k$ at iteration $k$ also lies in $\bar U_\Mcal(\tau)$ because $\|\hat x_k-\bar x_k\|\le\mathcal Z_k/\sqrt n\le\tau/2$. Applying the $2$-Lipschitz projection bound to the input and to the two averages gives
\begin{equation}\label{eq:local-displacement-bound}
\begin{aligned}
\|\mathbf x_{k+1}-\mathbf x_k\|&\le2c_i\ell_k,&
\|\bar{\mathbf x}_{k+1}-\bar{\mathbf x}_k\|&\le4c_i\ell_k,\\
\|T_{k+1}-T_k\|_{\rm op}&\le c_T\ell_k,&c_T&:=4L_{\Pcal}c_i/\sqrt n.
\end{aligned}
\end{equation}
Consequently, the auxiliary step bound is
\begin{equation}\label{eq:local-step-bound}
\begin{aligned}
\|\mathbf p_k-\mathbf x_k\|+\|\mathbf x_{k+1}-\mathbf x_k\|
+\|\bar{\mathbf x}_{k+1}-\bar{\mathbf x}_k\|
&\le s_p(\mathcal Z_k+\alpha G_k),\qquad s_p:=7c_i.
\end{aligned}
\end{equation}

\emph{Tangent and normal recursions.}
From \eqref{eq:analysis-notation}, $\mathbf d_j=(I-\mathbf J)\mathbf g_j$ for $j=k,k+1$. Applying \eqref{eq:quad} to each node gives $\|\mathbf g_{k+1}-\mathbf g_k\|^2=\sum_{i=1}^n\|\grad f_i(x_{i,k+1})-\grad f_i(x_{i,k})\|^2\le L^2\sum_{i=1}^n\|x_{i,k+1}-x_{i,k}\|^2=L^2\|\mathbf x_{k+1}-\mathbf x_k\|^2$. Since $I-\mathbf J$ is an orthogonal projector, \eqref{eq:local-displacement-bound} yields
\[
\begin{aligned}
\|\mathbf d_{k+1}-\mathbf d_k\|
&=\|(I-\mathbf J)(\mathbf g_{k+1}-\mathbf g_k)\|\\
&\le\|\mathbf g_{k+1}-\mathbf g_k\|
\le L\|\mathbf x_{k+1}-\mathbf x_k\|\\
&\le2Lc_i\ell_k=c_d\ell_k,\qquad c_d:=2Lc_i.
\end{aligned}
\]
Since $T_k$ commutes with $\mathbf W-\mathbf V$ and $T_k(\mathbf n_k)=0$, we have $T_k((\mathbf W-\mathbf V)\mathbf n_k)=0$. Using \eqref{eq:local-errors-def-split} and \eqref{eq:local-y-update}, we have
\begin{equation}\label{eq:local-r-recursion-split}
\begin{aligned}
\mathbf r_{k+1}
&=T_{k+1}(\mathbf y_{k+1}-\alpha\mathbf d_{k+1})\\
&=T_k(\mathbf y_{k+1}-\alpha\mathbf d_{k+1})
+(T_{k+1}-T_k)(\mathbf y_{k+1}-\alpha\mathbf d_{k+1})\\
&=\mathbf r_k+T_k\bigl((\mathbf W-\mathbf V)\mathbf e_k\bigr)
+T_k(\boldsymbol\gamma_k)-\alpha T_k(\mathbf d_{k+1}-\mathbf d_k)\\
&\quad+(T_{k+1}-T_k)(\mathbf y_{k+1}-\alpha\mathbf d_{k+1})\\
&=(\mathbf W-\mathbf V)\mathbf e_k+(I-\mathbf J)\mathbf r_k+\boldsymbol\epsilon_{r,k},\\
\boldsymbol\epsilon_{r,k}
&:=-(I-T_k)\bigl((\mathbf W-\mathbf V)\mathbf e_k\bigr)+T_k(\boldsymbol\gamma_k)
-\alpha T_k(\mathbf d_{k+1}-\mathbf d_k)\\
&\quad+(T_{k+1}-T_k)(\mathbf y_{k+1}-\alpha\mathbf d_{k+1}),
\end{aligned}
\end{equation}
where the last equality uses $\mathbf r_k=(I-\mathbf J)\mathbf r_k$ and $T_k((\mathbf W-\mathbf V)\mathbf e_k)=(\mathbf W-\mathbf V)\mathbf e_k-(I-T_k)((\mathbf W-\mathbf V)\mathbf e_k)$. For the first remainder term, commutation and \eqref{eq:local-current-bounds} give $\|(I-T_k)((\mathbf W-\mathbf V)\mathbf e_k)\|\le\beta\|(I-T_k)(\mathbf e_k)\|\le\beta\mathcal Z_k^2/(2R)$. The next two terms satisfy $\|T_k(\boldsymbol\gamma_k)\|\le\|\boldsymbol\gamma_k\|\le c_\gamma E_k$ and $\alpha\|T_k(\mathbf d_{k+1}-\mathbf d_k)\|\le c_d\alpha\ell_k$. For the last term, the bound $\|\mathbf y_{k+1}\|\le c_b(\alpha+\mathcal Z_k)$ above, $\|\mathbf d_{k+1}\|=\|(I-\mathbf J)\mathbf g_{k+1}\|\le D$, and \eqref{eq:local-displacement-bound} yield $\|(T_{k+1}-T_k)(\mathbf y_{k+1}-\alpha\mathbf d_{k+1})\|\le c_T\ell_k[c_b(\alpha+\mathcal Z_k)+\alpha D]\le c_T(c_b+D)\ell_k(\alpha+\mathcal Z_k)$. Since $G_k\le B$,
\[
\alpha\ell_k\le E_k+\alpha^2G_k,\qquad
\ell_k(\alpha+\mathcal Z_k)\le(1+B)E_k+\alpha^2G_k.
\]
Thus $\|\boldsymbol\epsilon_{r,k}\|\le c_rE_k+c_s\alpha^2G_k$, where
\begin{equation}\label{eq:tangent-error-constants}
\begin{aligned}
c_r&:=\frac\beta{2R}+c_\gamma+c_d+(1+B)c_T(c_b+D),\\
c_s&:=c_d+c_T(c_b+D).
\end{aligned}
\end{equation}
Similarly,
\begin{equation}\label{eq:local-n-recursion-split}
\begin{aligned}
\mathbf n_{k+1}&=\mathbf N\mathbf n_k+\boldsymbol\epsilon_{n,k},\\
\boldsymbol\epsilon_{n,k}
&:=(I-T_k)\bigl((\mathbf W-\mathbf V)\mathbf e_k\bigr)
+(I-T_k)(\boldsymbol\gamma_k)+(T_k-T_{k+1})(\mathbf y_{k+1}),
\end{aligned}
\end{equation}
so $\|\boldsymbol\epsilon_{n,k}\|\le c_nE_k+c_o\alpha^2G_k$, with
\begin{equation}\label{eq:normal-error-constants}
c_n:=\frac\beta{2R}+c_\gamma+(1+B)c_Tc_b,\qquad
c_o:=c_Tc_b.
\end{equation}
The linear part of the pair \eqref{eq:local-e-recursion-split}--\eqref{eq:local-r-recursion-split} is $\mathbf P$. Lemmas~\ref{lem:A} and~\ref{lem:local-P-stable-split}, \eqref{eq:explicit-norm-comparison}, and $\mathbf J\mathbf n_k=0$ prove \eqref{eq:local-Z-recur-split} with
\begin{equation}\label{eq:projection-recursion-constants}
A_p:=1+h_H(c_e+c_r)+c_n,\qquad
D_p:=1+h_Hc_s+c_o.
\end{equation}
The added $1$ in each definition ensures a positive denominator in the stepsize bounds below, including when some objective or geometric bounds vanish.
\end{proof}

\begin{lemma}\label{lem:local-bootstrap-split}
Fix $M>0$ and let $\mathcal U$ be the tube in \eqref{eq:smooth-tube}. Set $\rho:=(1+\nu)/2$ and define
\begin{equation}\label{eq:bootstrap-stepsize-explicit}
\alpha_b:=\min\left\{1,\frac{\delta_p}{1+M},
\frac{\rho-\nu}{A_p(1+M)},
\frac{(1-\rho)M}{2D_pB}\right\}.
\end{equation}
If $\hat x_0\in\mathcal U$, $\mathbf J\mathbf y_0=0$, $0<\alpha\le\alpha_b$, and $\mathcal Z_0\le M\alpha$, then $\mathbf p_k\in\mathcal U^n$ and $\hat x_k\in\mathcal U$ for every $k\ge0$, and
\begin{equation}\label{eq:local-Z-linear-split}
\mathcal Z_k\le M\alpha,\qquad
\mathcal Z_{k+1}\le\rho\mathcal Z_k+D_p\alpha^2G_k,\qquad k\ge0.
\end{equation}
\end{lemma}
\begin{proof}
If $\mathcal Z_k\le M\alpha$ and $\hat x_k\in\mathcal U$, then $(1+M)\alpha\le\delta_p$ allows Lemma~\ref{lem:local-recursion-split} to define the update at iteration $k+1$, with $\mathbf p_k\in\mathcal U^n$ and $\hat x_{k+1}\in\mathcal U$. The shifted recursion \eqref{eq:local-y-update} preserves $\mathbf J\mathbf y_{k+1}=0$. Since $\nu+A_p(1+M)\alpha\le\rho$, the error bound gives
\[
\mathcal Z_{k+1}\le\rho\mathcal Z_k+D_p\alpha^2G_k
\le\rho M\alpha+\frac{1-\rho}{2}M\alpha
=\frac{1+\rho}{2}M\alpha<M\alpha.
\]
Induction proves the assertions.
\end{proof}

We can now state the projection-based convergence and rate result in a single theorem.

\begin{theorem}[Projection-based REXTRA]\label{thm:local-conv}
Let $\Mcal$ be a compact smooth embedded submanifold, and fix $\theta\in(0,1/2]$ and a finite $\varepsilon>0$. Under Assumptions~\ref{assum:W} and~\ref{ass:lip}, let $\mathcal U$ be defined by \eqref{eq:smooth-tube}, and let $\alpha_0$ and $\Gamma_p$ be defined by \eqref{eq:projection-final-constants} below. Algorithm~\ref{alg}, initialized with
\[
\mathbf x_0\in\Mcal^n,\qquad
\mathbf s_0=-\alpha\mathbf g_0,\qquad
\hat x_0\in\mathcal U,\qquad
\|\mathbf x_0-\bar{\mathbf x}_0\|\le\varepsilon\alpha,
\]
is well defined for every $0<\alpha\le\alpha_0$, with $\mathbf p_k\in\mathcal U^n$ and $\hat x_k\in\mathcal U$ for every $k\ge0$. With $F_\star:=\min_{x\in\Mcal}F(x)$, for every integer $K\ge1$,
\begin{equation}\label{eq:projection-main-rate}
\min_{0\le k<K}\left(
\|\grad F(\bar x_k)\|^2+\|\mathbf x_k-\bar{\mathbf x}_k\|^2
\right)
\le\frac{\Gamma_p}K
\left(\frac{F(\bar x_0)-F_\star}{\alpha}+\mathcal Z_0^2\right).
\end{equation}
The displayed constants are independent of $K$ and of $\alpha\in(0,\alpha_0]$.
\end{theorem}

\begin{proof}
At initialization, $\mathbf y_0=\mathbf n_0=0$ and $\mathbf r_0=-\alpha T_0(\mathbf d_0)$. Therefore
\[
\mathcal Z_0\le h_H(\|\mathbf e_0\|+\alpha D)
\le h_H(\varepsilon+D)\alpha<M\alpha,
\qquad M:=1+h_H(\varepsilon+D).
\]
Use this $M$ in \eqref{eq:bootstrap-stepsize-explicit}. Define
\begin{equation}\label{eq:projection-descent-constants}
\begin{aligned}
q&:=\frac{1+\rho^2}{2},&a&:=1-q=\frac{1-\rho^2}{2},\\
u_p&:=m_p(1+M),&v_p&:=j_p,\\
C_f&:=(1+L)u_p^2,&
C_z&:=D_p^2\frac{1+\rho^2}{1-\rho^2}.
\end{aligned}
\end{equation}
The stepsize and rate constant in the theorem are
\begin{equation}\label{eq:projection-final-constants}
\begin{aligned}
\alpha_0&:=\min\left\{\alpha_b,
\frac1{4[1+v_pB+L(1+v_pB)^2]},
\left(\frac{a}{4C_fC_z}\right)^{1/4}\right\},\\
\Gamma_p&:=\max\left\{4\left(1+\frac{C_z}a\right),
\frac{4C_f(1+C_z/a)+1}{a}\right\}.
\end{aligned}
\end{equation}
Lemma~\ref{lem:local-bootstrap-split} ensures $\mathcal Z_k\le M\alpha$ and all the one-step bounds for every $k$.

\emph{Objective descent.}
By \eqref{eq:local-center-displacement-remainder-split} and $E_k\le(1+M)\alpha\mathcal Z_k$,
\[
\bar x_{k+1}-\bar x_k=-\alpha\grad F(\bar x_k)+\omega_k,\qquad
\|\omega_k\|\le u_p\alpha\mathcal Z_k+v_p\alpha^2G_k^2.
\]
The first inequality in \eqref{eq:quad} gives
\[
F(\bar x_{k+1})\le F(\bar x_k)-\alpha G_k^2
+G_k\|\omega_k\|+\frac L2(\alpha G_k+\|\omega_k\|)^2.
\]
Specifically,
\[
\begin{aligned}
G_k\|\omega_k\|&\le\frac\alpha4G_k^2+u_p^2\alpha\mathcal Z_k^2
+v_pB\alpha^2G_k^2,\\
(\alpha G_k+\|\omega_k\|)^2
&\le2\alpha^2\bigl[(1+v_pB)^2G_k^2+u_p^2\mathcal Z_k^2\bigr].
\end{aligned}
\]
The second restriction in \eqref{eq:projection-final-constants} and $\alpha\le1$ consequently yield
\begin{equation}\label{eq:local-center-descent-split}
F(\bar x_{k+1})\le F(\bar x_k)-\frac\alpha2G_k^2+C_f\alpha\mathcal Z_k^2.
\end{equation}

\emph{Summation and rate.}
Young's inequality applied to \eqref{eq:local-Z-linear-split} gives
\begin{equation}\label{eq:local-Z-square-split}
\mathcal Z_{k+1}^2\le q\mathcal Z_k^2+C_z\alpha^4G_k^2,
\end{equation}
with exactly the $q$ and $C_z$ in \eqref{eq:projection-descent-constants}. Summing \eqref{eq:local-Z-square-split} and \eqref{eq:local-center-descent-split} from $k=0$ to $K-1$ gives
\begin{equation}\label{eq:local-partial-sums}
\begin{aligned}
a\sum_{k=0}^{K-1}\mathcal Z_k^2
&\le\mathcal Z_0^2+C_z\alpha^4\sum_{k=0}^{K-1}G_k^2,\\
\frac\alpha2\sum_{k=0}^{K-1}G_k^2
&\le F(\bar x_0)-F_\star+C_f\alpha\sum_{k=0}^{K-1}\mathcal Z_k^2.
\end{aligned}
\end{equation}
The last restriction in \eqref{eq:projection-final-constants} ensures $C_fC_z\alpha^4/a\le1/4$. Substitution yields
\begin{equation}\label{eq:local-uniform-sums}
\begin{aligned}
\sum_{k=0}^{K-1}G_k^2
&\le\frac4\alpha\bigl(F(\bar x_0)-F_\star\bigr)
+\frac{4C_f}a\mathcal Z_0^2,\\
\sum_{k=0}^{K-1}\mathcal Z_k^2
&\le\frac1a\mathcal Z_0^2+\frac{C_z\alpha^4}a\sum_{k=0}^{K-1}G_k^2.
\end{aligned}
\end{equation}
Since $\|\mathbf e_k\|^2\le\mathcal Z_k^2$ and $\alpha\le1$, the sum of the two squared residuals is at most
\[
4\left(1+\frac{C_z}a\right)\frac{F(\bar x_0)-F_\star}{\alpha}
+\frac{4C_f(1+C_z/a)+1}{a}\mathcal Z_0^2.
\]
Averaging proves \eqref{eq:projection-main-rate} with the stated $\Gamma_p$. The right-hand sides of \eqref{eq:local-uniform-sums} are uniformly bounded in $K$, so letting $K\to\infty$ gives $\sum_{k=0}^{\infty}G_k^2<\infty$ and $\sum_{k=0}^{\infty}\mathcal Z_k^2<\infty$. Hence $G_k$ and $\mathcal Z_k$ tend to zero. Compactness supplies accumulation points of the mean sequence, and continuity of $\grad F$ proves their stationarity.
\end{proof}

\subsection{Retraction-based REXTRA}\label{sec-retraction}
We next transfer the one-step estimates to \eqref{eq:retrac-R-extra}. The comparison with metric projection is used only in the proof; the retraction-based update evaluates no metric projection.

\begin{assumption}\label{ass:retr-C2}
The retraction $\Rcal$ is $C^2$ on a neighborhood of the zero section of $T\Mcal$. For every $x\in\Mcal$, $\Rcal_x(0)=x$ and $D\Rcal_x(0)[v]=v$ for every $v\in T_x\Mcal$.
\end{assumption}

\begin{lemma}\label{lem:retr-projection-agreement}
Choose $\tau_r\in(0,\tau/2]$ so that $\Rcal_x(v)$ is defined on a neighborhood of every tangent disk $\{v\in T_x\Mcal:\|v\|\le\tau_r\}$. Define
\begin{equation}\label{eq:retraction-agreement-constant}
\begin{aligned}
\Psi(x,v)&:=\Pcal_\Mcal(x+v)-\Rcal_x\bigl(\Pcal_{T_x\Mcal}(v)\bigr),\\
C_r&:=\max\left\{1,\frac12\sup_{\substack{x\in\Mcal\\\|v\|\le\tau_r}}
\|D_v^2\Psi(x,v)\|_{\rm op}\right\}.
\end{aligned}
\end{equation}
Then, for $x\in\Mcal$ and $\|p-x\|\le\tau_r$,
\begin{equation}\label{eq:retr-projection-agreement}
\left\|\Rcal_x\bigl(\Pcal_{T_x\Mcal}(p-x)\bigr)-\Pcal_\Mcal(p)\right\|
\le C_r\|p-x\|^2.
\end{equation}
Both maps are defined on this domain. The same bound, with this $C_r$, holds blockwise on $\Mcal^n$.
\end{lemma}
\begin{proof}
Compactness and Assumption~\ref{ass:retr-C2} give such a uniform $\tau_r$ and make the supremum in \eqref{eq:retraction-agreement-constant} finite. For the map $\Psi$ there, $\Psi(x,0)=0$ and
$D_v\Psi(x,0)=\Pcal_{T_x\Mcal}-D\Rcal_x(0)\Pcal_{T_x\Mcal}=0$.
Taylor's theorem therefore gives \eqref{eq:retr-projection-agreement} with the defined $C_r$. For the product statement, square the block bounds and use $\sum_i\|p_i-x_i\|^4\le(\sum_i\|p_i-x_i\|^2)^2$.
\end{proof}

\begin{lemma}\label{lem:retr-local-transfer}
Under Assumption~\ref{ass:retr-C2}, the conclusions of Lemma~\ref{lem:local-recursion-split} hold for the retraction update with the same $\nu$ and the following fixed replacements:
\begin{equation}\label{eq:retraction-transfer-constants}
\begin{aligned}
d_r&:=2C_rc_i^2,&
t_T&:=2L_{\Pcal}/\sqrt n,\\
t_r&:=t_T(c_b+D)+L,&
t_n&:=t_Tc_b,\\
t_Z&:=h_H(3+t_r)+t_n,&
A_r&:=A_p+t_Zd_r,\\
D_r&:=D_p+t_Zd_rB,&
m_r&:=m_p+2d_r/\sqrt n,\\
j_r&:=j_p+2d_r/\sqrt n.
\end{aligned}
\end{equation}
The radius is
\begin{equation}\label{eq:retraction-radius-explicit}
\delta_r:=\min\left\{\delta_p,\frac{\tau_r}{c_iB},
\frac{\tau\sqrt n}{2[c_\zeta+d_r(1+B^2)]}\right\}.
\end{equation}
In particular, the bounds are $\mathcal Z_{k+1}\le\nu\mathcal Z_k+A_rE_k+D_r\alpha^2G_k$,
$\|\bar x_{k+1}-\bar x_k+\alpha\grad F(\bar x_k)\|\le m_rE_k+j_r\alpha^2G_k^2$. Every retraction input lies in its prescribed tangent disk.
\end{lemma}
\begin{proof}
Keep $E_k=(\alpha+\mathcal Z_k)\mathcal Z_k$ and $\ell_k=\mathcal Z_k+\alpha G_k$. The input $\mathbf p_k$ and the shifted update \eqref{eq:local-y-update} depend only on the variables at iteration $k$ and are common to both methods. The input estimate in the proof of Lemma~\ref{lem:local-recursion-split} gives
\[
\|\mathbf p_k-\mathbf x_k\|\le c_i\ell_k
\le c_iB(\alpha+\mathcal Z_k)\le\tau_r.
\]
Hence $Q_k(\mathbf p_k-\mathbf x_k)$ is in the retraction domain. Define the one-step projected comparison and the actual retraction step by
\[
\begin{aligned}
\widetilde{\mathbf x}_{k+1}&:=\Pcal_{\Mcal^n}(\mathbf p_k),\\
\mathbf x_{k+1}&:=\Rcal_{\mathbf x_k}\bigl(Q_k(\mathbf p_k-\mathbf x_k)\bigr),\qquad
\boldsymbol\delta_k:=\mathbf x_{k+1}-\widetilde{\mathbf x}_{k+1}.
\end{aligned}
\]
Lemma~\ref{lem:retr-projection-agreement} gives
\begin{equation}\label{eq:retr-discrepancy-bound}
\|\boldsymbol\delta_k\|\le C_rc_i^2\ell_k^2
\le d_r\bigl(\mathcal Z_k^2+\alpha^2G_k^2\bigr).
\end{equation}
Lemma~\ref{lem:local-recursion-split} applies to the projected comparison because its hypotheses concern iteration $k$ and do not require a history generated by projection updates. Thus
$\|\widetilde{\mathbf x}_{k+1}-\boldsymbol\zeta_k\|\le c_\zeta\mathcal Z_k$. Moreover,
\[
\|\mathbf x_{k+1}-\boldsymbol\zeta_k\|
\le c_\zeta\mathcal Z_k+d_r(\mathcal Z_k^2+\alpha^2G_k^2)
\le[c_\zeta+d_r(1+B^2)]\delta_r\le\tau\sqrt n/2.
\]
Both averages at iteration $k+1$, and the segment between them, are consequently within distance $\tau/2$ of $z_k\in\Mcal$. Their centers are defined before the Lipschitz estimates below are used.

Use tildes for quantities formed from $\widetilde{\mathbf x}_{k+1}$, including its center and its tangent projector. The shifted state $\mathbf y_{k+1}$ is common to both updates. Set
$\widetilde{\mathbf s}_{k+1}:=\mathbf y_{k+1}-\alpha\grad f(\widetilde{\mathbf x}_{k+1})$ and define the other comparison quantities by \eqref{eq:analysis-notation}--\eqref{eq:local-Z-def-split}. The $2$-Lipschitz center map and \eqref{eq:quad} give
\[
\begin{aligned}
\|\bar{\mathbf x}_{k+1}-\widetilde{\bar{\mathbf x}}_{k+1}\|
&\le2\|\boldsymbol\delta_k\|,\\
\|\mathbf d_{k+1}-\widetilde{\mathbf d}_{k+1}\|
&\le L\|\boldsymbol\delta_k\|,\\
\|T_{k+1}-\widetilde T_{k+1}\|_{\rm op}
&\le t_T\|\boldsymbol\delta_k\|.
\end{aligned}
\]
Here $\|\mathbf y_{k+1}\|\le c_b(\alpha+\mathcal Z_k)\le c_b$ and both gradient discrepancies at iteration $k+1$ have norm at most $D$. Therefore
\begin{equation}\label{eq:retr-one-step-transfer}
\begin{aligned}
\|\mathbf e_{k+1}-\widetilde{\mathbf e}_{k+1}\|&\le3\|\boldsymbol\delta_k\|,\\
\|\mathbf r_{k+1}-\widetilde{\mathbf r}_{k+1}\|&\le t_r\|\boldsymbol\delta_k\|,\\
\|\mathbf n_{k+1}-\widetilde{\mathbf n}_{k+1}\|&\le t_n\|\boldsymbol\delta_k\|.
\end{aligned}
\end{equation}
For the middle bound, expand $T_{k+1}(\mathbf y_{k+1}-\alpha\mathbf d_{k+1})-\widetilde T_{k+1}(\mathbf y_{k+1}-\alpha\widetilde{\mathbf d}_{k+1})$ and use $\alpha\le1$. The two auxiliary corrections $\mathbf s_{k+1}$ and $\widetilde{\mathbf s}_{k+1}$ need not agree.

By \eqref{eq:explicit-norm-comparison}, $\mathcal Z_{k+1}\le\widetilde{\mathcal Z}_{k+1}+t_Z\|\boldsymbol\delta_k\|$. Combining this with the projected recursion and \eqref{eq:retr-discrepancy-bound}, using $\mathcal Z_k^2\le E_k$ and $G_k^2\le BG_k$, yields the stated recursion with $A_r$ and $D_r$.
For the center bound, retain $G_k^2$ in \eqref{eq:retr-discrepancy-bound}. The center comparison then gives
\begin{equation}\label{eq:retr-center-remainder}
\|\bar x_{k+1}-\bar x_k+\alpha\grad F(\bar x_k)\|
\le m_rE_k+j_r\alpha^2G_k^2.
\end{equation}
Finally, the projected step bound \eqref{eq:local-step-bound} changes by at most $3\|\boldsymbol\delta_k\|$. Since $\ell_k\le B\delta_r\le B$, its retraction counterpart holds with $s_r:=s_p+3C_rc_i^2B$. The exact shifted update preserves $\mathbf J\mathbf y_{k+1}=0$.
\end{proof}

\begin{theorem}[Retraction-based REXTRA]\label{thm:retr-global-conv}
Let $\Mcal$ be a compact smooth embedded submanifold. Suppose Assumptions~\ref{assum:W}, \ref{ass:lip}, and~\ref{ass:retr-C2} hold, and fix $\theta\in(0,1/2]$. Let $\mathcal U$ be the tube in \eqref{eq:smooth-tube}. For any fixed finite $\varepsilon>0$, let $\alpha_r$ and $\Gamma_r$ be defined by \eqref{eq:retraction-final-constants} below. Algorithm~\ref{alg-retrac}, initialized with
\[
\mathbf x_0\in\Mcal^n,\qquad
\mathbf s_0=-\alpha\mathbf g_0,\qquad
\hat x_0\in\mathcal U,\qquad
\|\mathbf x_0-\bar{\mathbf x}_0\|\le\varepsilon\alpha,
\]
is well defined for $0<\alpha\le\alpha_r$, its center averages remain in $\mathcal U$, and its retraction inputs remain in the prescribed tangent disks. For every integer $K\ge1$,
\begin{equation}\label{eq:retr-global-rate}
\min_{0\le k<K}\left(
\|\grad F(\bar x_k)\|^2+\|\mathbf x_k-\bar{\mathbf x}_k\|^2
\right)
\le\frac{\Gamma_r}K
\left(\frac{F(\bar x_0)-F_\star}{\alpha}+\mathcal Z_0^2\right).
\end{equation}
For these iterates, $\sum_{k=0}^{\infty}\|\grad F(\bar x_k)\|^2<\infty$ and $\sum_{k=0}^{\infty}\mathcal Z_k^2<\infty$. Hence the consensus and stationarity residuals vanish, and every accumulation point of $\{\bar x_k\}$ is stationary. The displayed constants are independent of $K$ and of $\alpha\in(0,\alpha_r]$.

In particular, every consensual feasible initialization $\mathbf x_0=\mathbf1_n\otimes x^\circ$, with $x^\circ\in\Mcal$ and $\mathbf s_0=-\alpha\mathbf g_0$, satisfies these initialization conditions. The constants can be chosen uniformly over $x^\circ\in\Mcal$.
\end{theorem}
\begin{proof}
Keep $M=1+h_H(\varepsilon+D)$, $\rho=(1+\nu)/2$, $q=(1+\rho^2)/2$, and $a=1-q$ as in the projection proof. Define
\begin{equation}\label{eq:retraction-descent-constants}
\begin{aligned}
\widetilde\alpha_b&:=\min\left\{1,\frac{\delta_r}{1+M},
\frac{\rho-\nu}{A_r(1+M)},\frac{(1-\rho)M}{2D_rB}\right\},\\
u_r&:=m_r(1+M),\qquad v_r:=j_r,\\
\widetilde{C}_f&:=(1+L)u_r^2,\qquad
\widetilde{C}_z:=D_r^2\frac{1+\rho^2}{1-\rho^2}.
\end{aligned}
\end{equation}
The constants in the theorem are
\begin{equation}\label{eq:retraction-final-constants}
\begin{aligned}
\alpha_r&:=\min\left\{\widetilde\alpha_b,
\frac1{4[1+v_rB+L(1+v_rB)^2]},
\left(\frac a{4\widetilde{C}_f\widetilde{C}_z}\right)^{1/4}\right\},\\
\Gamma_r&:=\max\left\{4\left(1+\frac{\widetilde{C}_z}a\right),
\frac{4\widetilde{C}_f(1+\widetilde{C}_z/a)+1}{a}\right\}.
\end{aligned}
\end{equation}
The initialization bound is $\mathcal Z_0\le h_H(\varepsilon+D)\alpha<M\alpha$. Applying the induction in Lemma~\ref{lem:local-bootstrap-split} with the explicitly defined $\delta_r,A_r,D_r,\widetilde\alpha_b$ gives
\[
\mathcal Z_k\le M\alpha,\qquad
\mathcal Z_{k+1}\le\rho\mathcal Z_k+D_r\alpha^2G_k.
\]
The center averages and retraction inputs remain in their stated domains. By \eqref{eq:retr-center-remainder},
\[
\bar x_{k+1}-\bar x_k=-\alpha\grad F(\bar x_k)+\omega_k,\qquad
\|\omega_k\|\le u_r\alpha\mathcal Z_k+v_r\alpha^2G_k^2.
\]
The same two Young inequalities displayed in the projection proof, with $u_r,v_r$ in place of $u_p,v_p$, now give
\[
\begin{aligned}
F(\bar x_{k+1})&\le F(\bar x_k)-\frac\alpha2G_k^2+\widetilde{C}_f\alpha\mathcal Z_k^2,\\
\mathcal Z_{k+1}^2&\le q\mathcal Z_k^2+\widetilde{C}_z\alpha^4G_k^2.
\end{aligned}
\]
The restrictions in \eqref{eq:retraction-final-constants} ensure these estimates and $\widetilde{C}_f\widetilde{C}_z\alpha^4/a\le1/4$. Summing and eliminating the error sum exactly as in \eqref{eq:local-partial-sums} yields
\[
\begin{aligned}
\sum_{k=0}^{K-1}G_k^2
&\le\frac4\alpha(F(\bar x_0)-F_\star)+\frac{4\widetilde{C}_f}a\mathcal Z_0^2,\\
\sum_{k=0}^{K-1}\mathcal Z_k^2
&\le\frac1a\mathcal Z_0^2+\frac{\widetilde{C}_z\alpha^4}a\sum_{k=0}^{K-1}G_k^2.
\end{aligned}
\]
Together with $\|\mathbf e_k\|^2\le\mathcal Z_k^2$ and $\alpha\le1$, these bounds prove \eqref{eq:retr-global-rate} with the defined $\Gamma_r$, summability, and the stated limiting properties.

For a consensual feasible initialization $\mathbf x_0=\mathbf1_n\otimes x^\circ$, we have
\[
\hat x_0=\bar x_0=x^\circ\in\Mcal,\qquad
\mathbf y_0=0.
\]
Thus $\mathbf e_0=\mathbf n_0=0$ and $\mathcal Z_0\le h_HD\alpha$, uniformly in $x^\circ$. The preceding hypotheses therefore hold with constants independent of the common feasible point.
\end{proof}

The two theorems establish decay of the residuals and stationarity of accumulation points. Convergence of the entire mean sequence to a single stationary point requires further assumptions. The small-stepsize and local-consensus conditions are sufficient conditions for the invariant neighborhood used in both proofs; a consensual feasible initialization satisfies the latter condition exactly.

\section{Numerical experiments} \label{sec-num}
{In this section, we compare projection-based REXTRA (Algorithm~\ref{alg}) and the retraction-based optimization update in Algorithm~\ref{alg-retrac} with DRDGD \cite{chen2021decentralized}, DPRGD \cite{deng2023decentralized}, DRGTA \cite{chen2021decentralized}, and DPRGT \cite{deng2023decentralized} on decentralized principal component analysis, distributed generalized eigenvalue problems, and low-rank matrix completion.}

\subsection{Decentralized principal component analysis} \label{sec-PCA}

The decentralized principal component analysis (PCA) problem seeks a subspace for dimensionality reduction that preserves the maximal variation of data distributed across multiple nodes. Mathematically, it can be formulated as
\begin{equation}
\label{eq:dpca}
\min_{\bx \in \mathcal{M}^n} -\frac{1}{2n} \sum_{i=1}^n \mathrm{tr}(x_i^\top A_i^\top A_i x_i), \quad \text{s.t.} \quad x_1 = \cdots = x_n,
\end{equation}
where $
\mathcal{M}^n := \mathrm{St}(d, r)^n$,
$\mathrm{St}(d, r)$ denotes the Stiefel manifold, $n$ is the number of nodes, and $A_i \in \mathbb{R}^{m_i \times d}$ represents the local data matrix stored at node $i$ with $m_i$ samples. {The negative sign appears because PCA maximizes the explained variance, while \eqref{eq:dpca} is written in the minimization form used throughout the paper.} Note that for any solution $x^*$ to \eqref{eq:dpca}, the transformed point $x^*Q$ is also a solution for any orthogonal matrix $Q \in \mathbb{R}^{r \times r}$. To properly measure the distance between two points $x$ and $x^*$, we define $d_s(x, x^*) := \min_{\substack{Q \in \mathbb{R}^{r \times r},\; Q^\top Q = QQ^\top = I_r}} \|xQ - x^*\|$.

{\par\medskip\noindent\textbf{Synthetic dataset.}}
{We fix $m_1=\cdots=m_n=1000$, $d=10$, and $r=5$. We generate a matrix $G_{\rm syn}\in\mathbb R^{1000n\times d}$ with independent standard Gaussian entries and compute its reduced singular value decomposition
$G_{\rm syn}=U\Sigma V^\top$, where $U\in\mathbb R^{1000n\times d}$ has orthonormal columns, $V\in\mathbb R^{d\times d}$ is orthogonal, and $\Sigma\in\mathbb R^{d\times d}$ is diagonal. To prescribe a decreasing singular-value profile, for a fixed $\xi\in(0,1)$ we set
$\widetilde\Sigma:=\operatorname{diag}(\xi^j)_{j=0}^{d-1}$. We then construct}
$
A = U \tilde{\Sigma} V^\top \in \mathbb{R}^{1000n \times d},
$
{and split the rows of $A$ uniformly at random into $n$ disjoint subsets $\{A_i\}$ of equal size. By construction, the first $r$ columns of $V$ form a solution to~\eqref{eq:dpca}. In our experiments, we set $\xi=0.8$ and $n=8$. We set the maximum number of epochs to 2000 and terminate early if $\|\grad F(\bar{x}_k)\|<10^{-8}$. All algorithms use constant stepsizes. We perform a grid search over $\{1,2,4,6,8\}\times\{10^{-5},10^{-4},10^{-3},10^{-2}\}$ to select the best stepsize for each algorithm.}

{We first examine the projection- and retraction-based methods in Algorithms~\ref{alg} and~\ref{alg-retrac}, denoted by REXTRA and REXTRA-R, across different network topologies. We consider Erd\H{o}s--R\'{e}nyi (ER) graphs with connection probabilities $p=0.3$ and $p=0.6$, together with the ring graph used in \cite{shi2015extra}. Figure~{\ref{fig:SD1}} shows that denser connectivity generally accelerates convergence. The two variants exhibit closely related behavior in these tests, consistent with the local second-order agreement established in Lemma~\ref{lem:retr-projection-agreement}. Motivated by this observation, the remaining experiments focus on the projection-based variant, denoted simply by REXTRA.}

{We then compare REXTRA with DRDGD, DPRGD, DRGTA, and DPRGT, using the ER graph with $p=0.6$ for all methods. Figures~{\ref{fig:SD2}} and~{\ref{fig:SD3}} report performance against the total number of communicated scalar entries and the number of epochs, respectively. REXTRA reaches the target accuracy with the smallest communication budget among the tested methods, benefiting from both its low per-iteration communication volume and its relatively large admissible stepsizes. Figure~{\ref{fig:dd1}} further shows that DPRGT becomes unstable at the larger tested stepsizes, whereas REXTRA remains stable.}

\begin{figure}[htp]
    \centering
    {\includegraphics[width=0.24\textwidth]{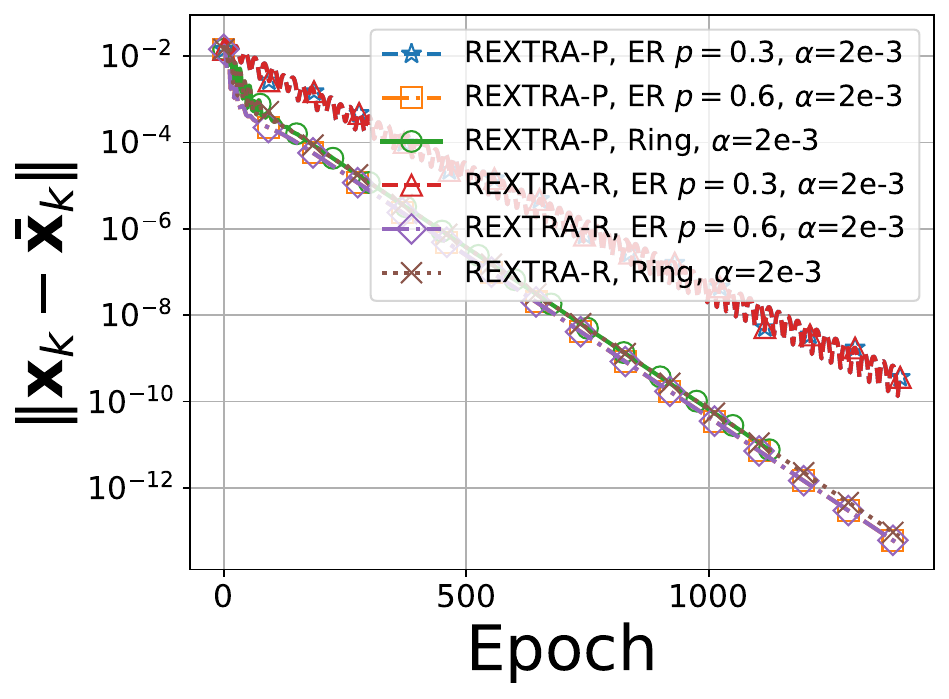}}
    \hfill
    {\includegraphics[width=0.24\textwidth]{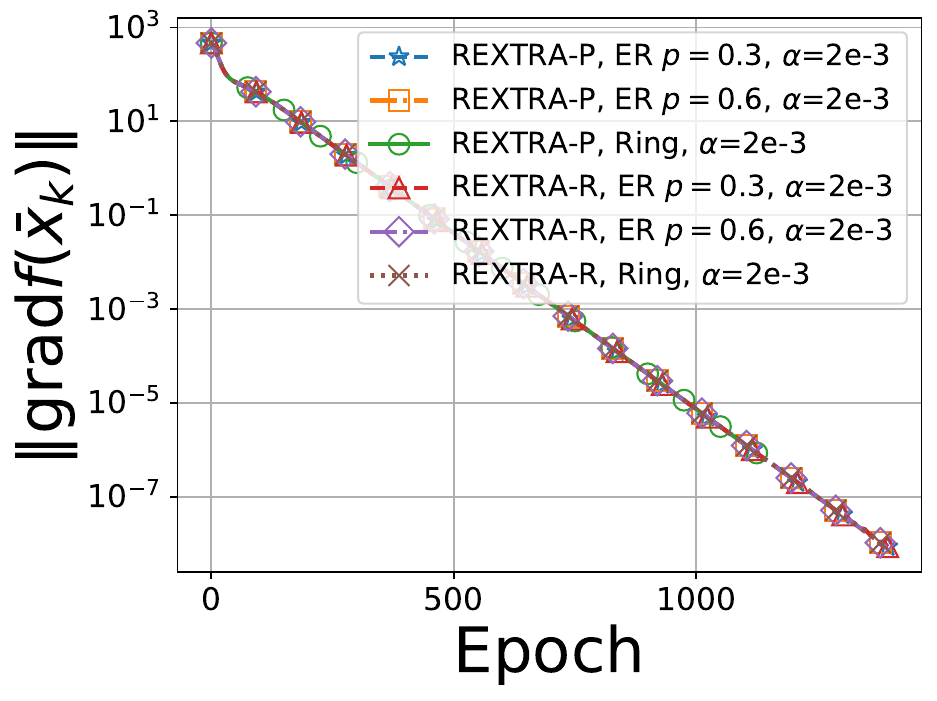}}
    \hfill
    {\includegraphics[width=0.24\textwidth]{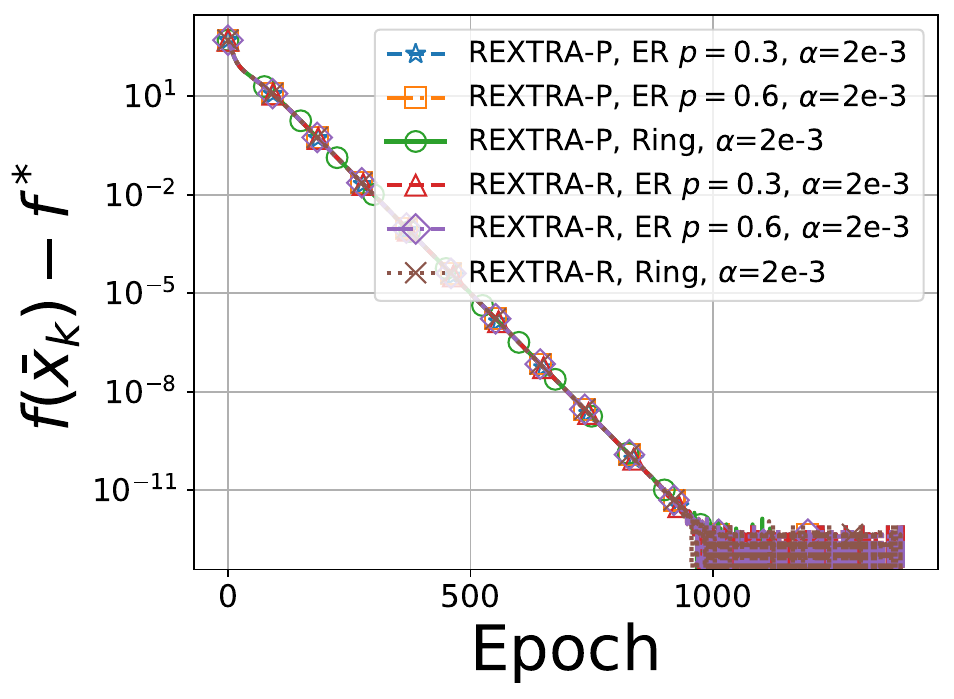}}
    {\includegraphics[width=0.24\textwidth]{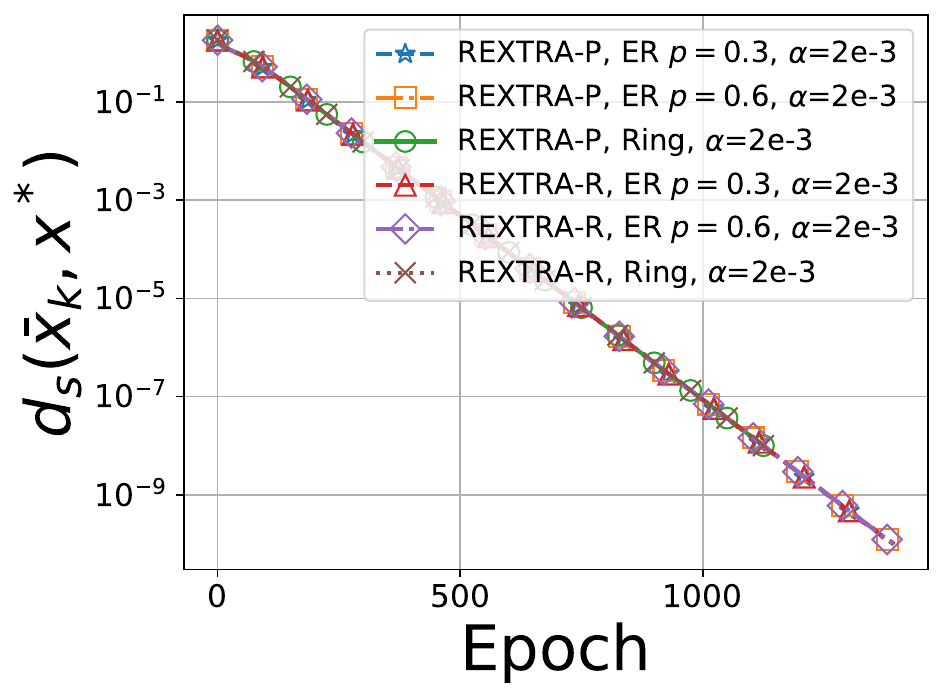}}
    \caption{Results of the REXTRA variants on different graphs.}
    \label{fig:SD1}
\end{figure}

\begin{figure}[htp]
    \centering
    {\includegraphics[width=0.24\textwidth]{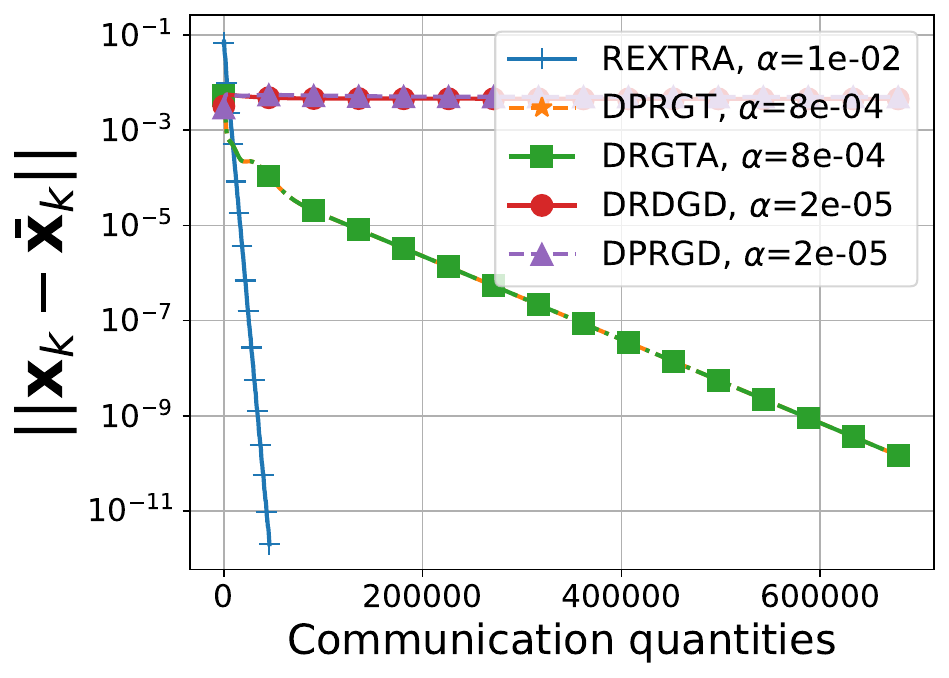}}
    \hfill
    {\includegraphics[width=0.24\textwidth]{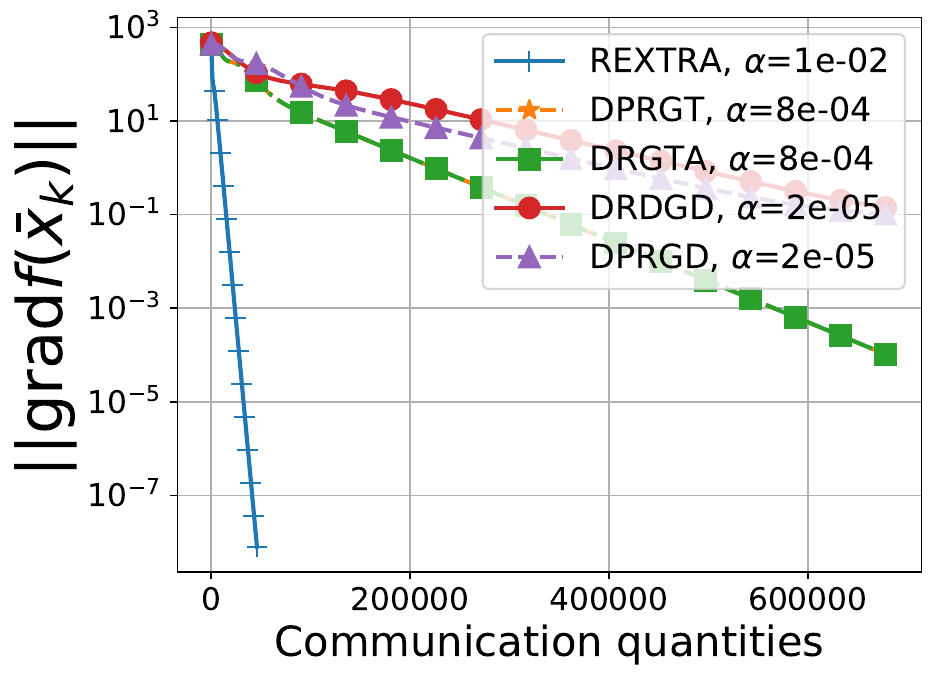}}
    \hfill
    {\includegraphics[width=0.24\textwidth]{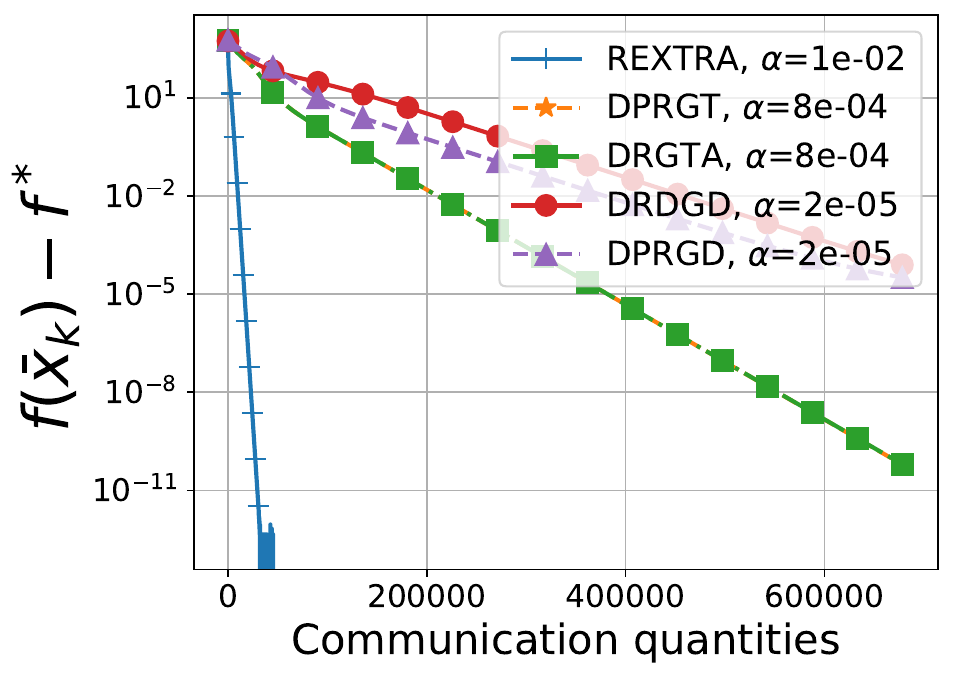}}
    {\includegraphics[width=0.24\textwidth]{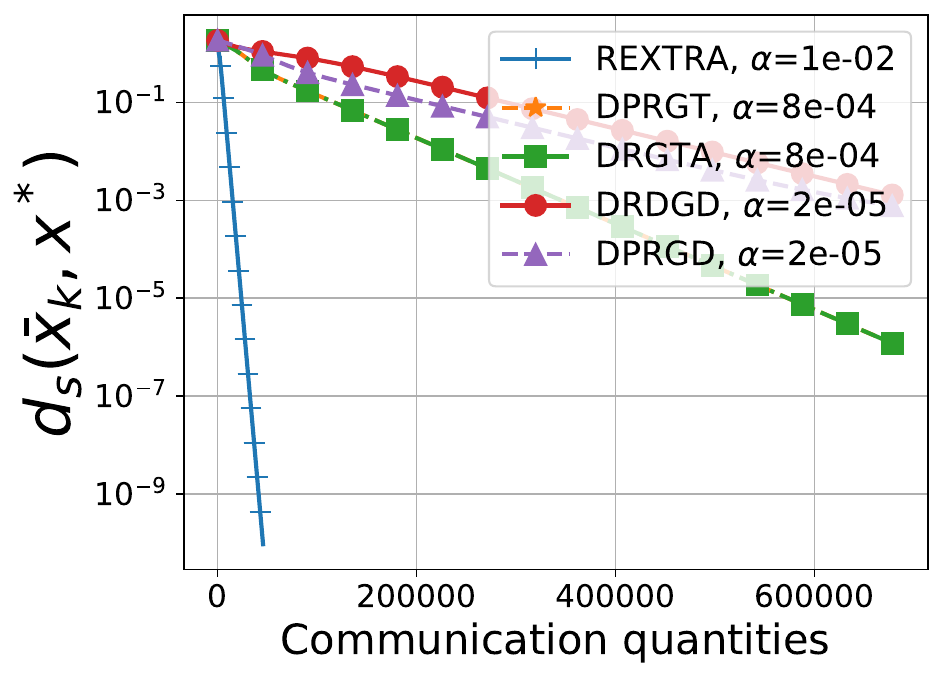}}
    \caption{Results of the tested algorithms for synthetic PCA, measured by communicated scalar entries.}
    \label{fig:SD2}
\end{figure}

\begin{figure}[htp]
    \centering
    {\includegraphics[width=0.24\textwidth]{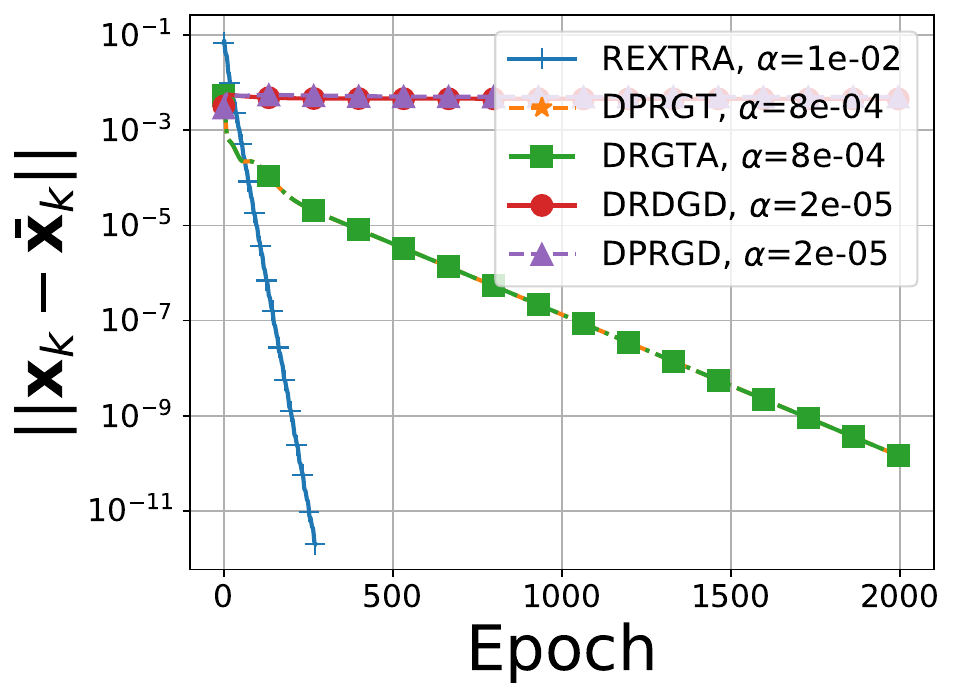}}
    \hfill
    {\includegraphics[width=0.24\textwidth]{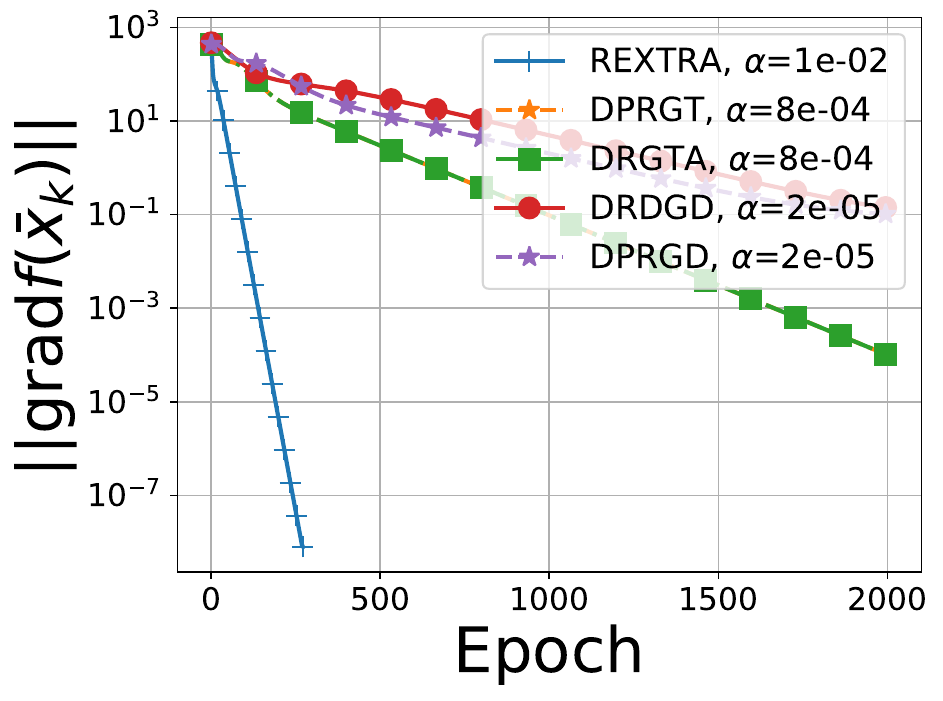}}
    \hfill
    {\includegraphics[width=0.24\textwidth]{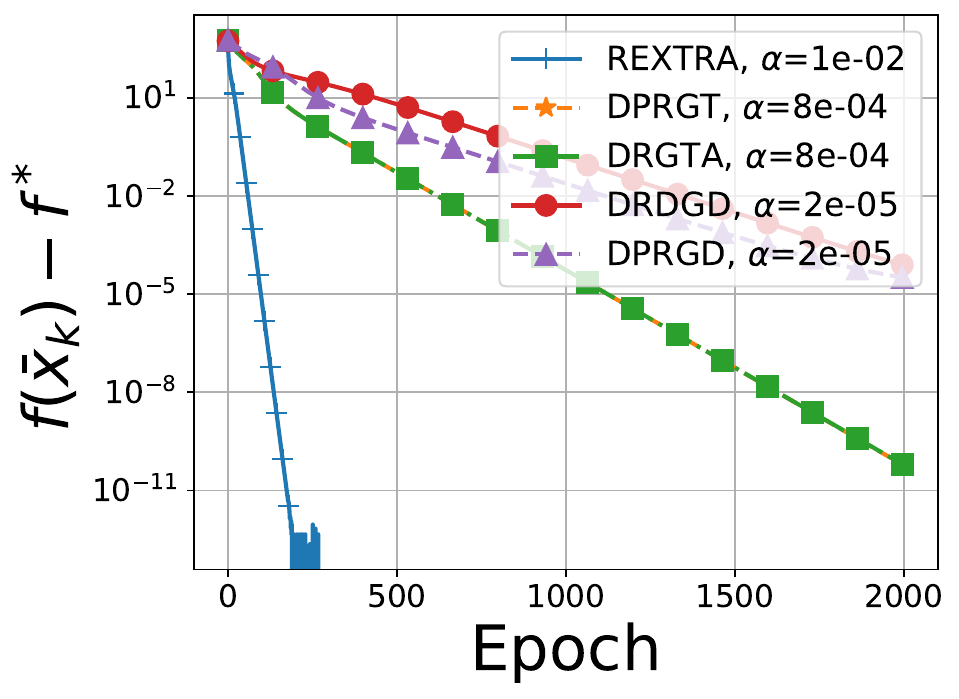}}
    {\includegraphics[width=0.24\textwidth]{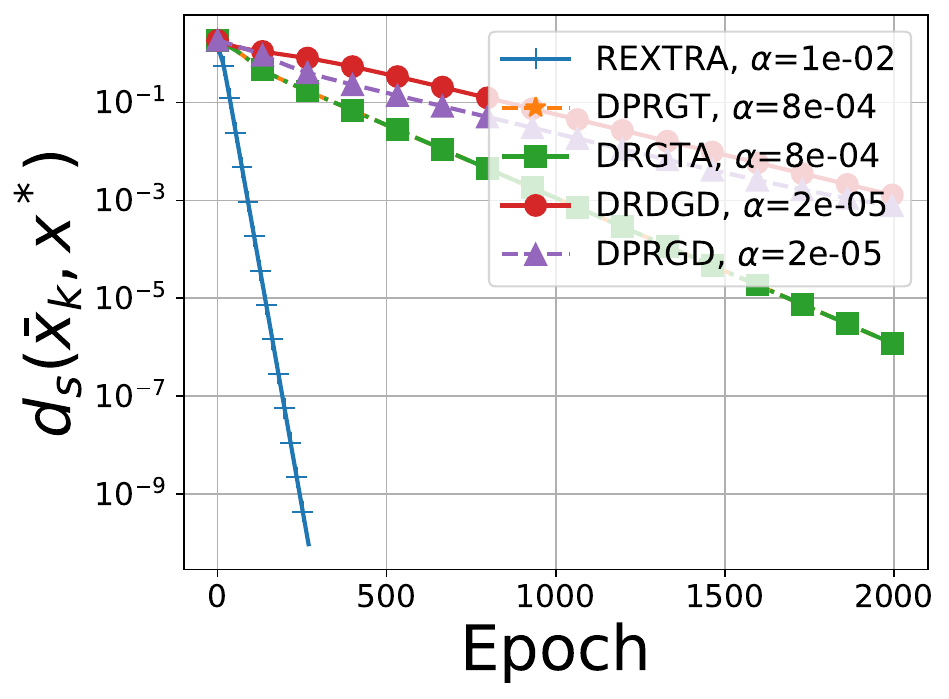} }
    \caption{Results of the tested algorithms for the PCA problem on the synthetic dataset, measured in epochs.}
    \label{fig:SD3}
\end{figure}

\begin{figure}[htp]
    \centering
    {\includegraphics[width=0.24\textwidth]{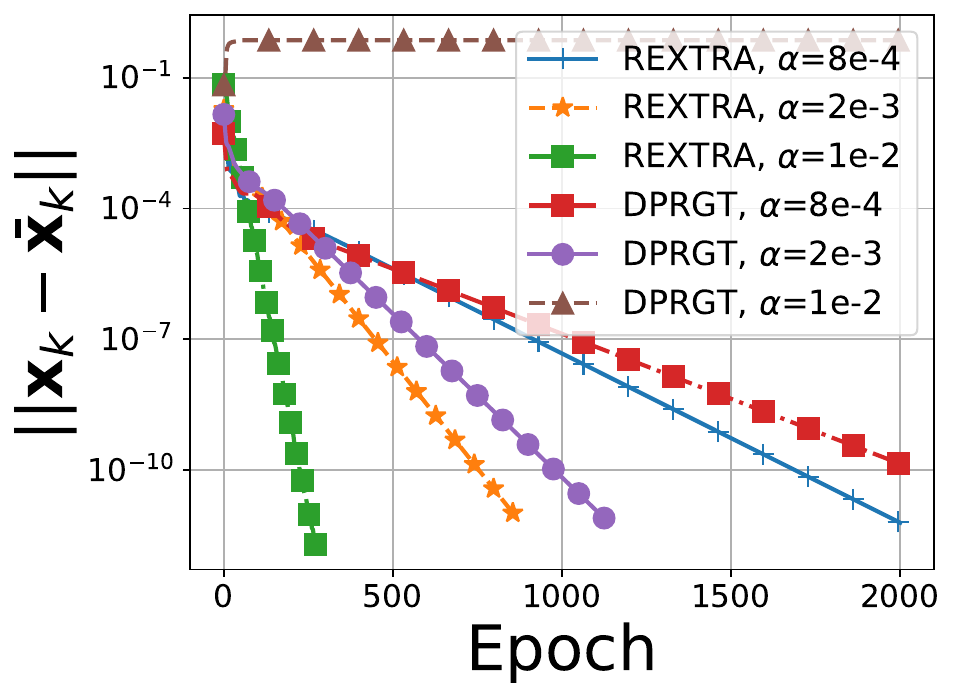}}
    \hfill
    {\includegraphics[width=0.24\textwidth]{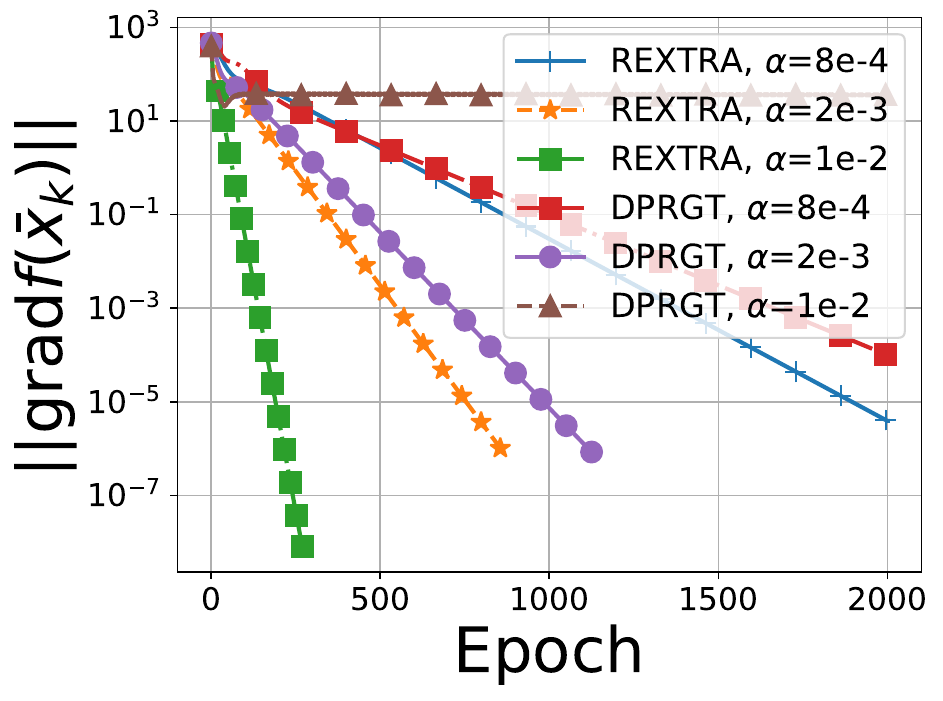}}
    \hfill
    {\includegraphics[width=0.24\textwidth]{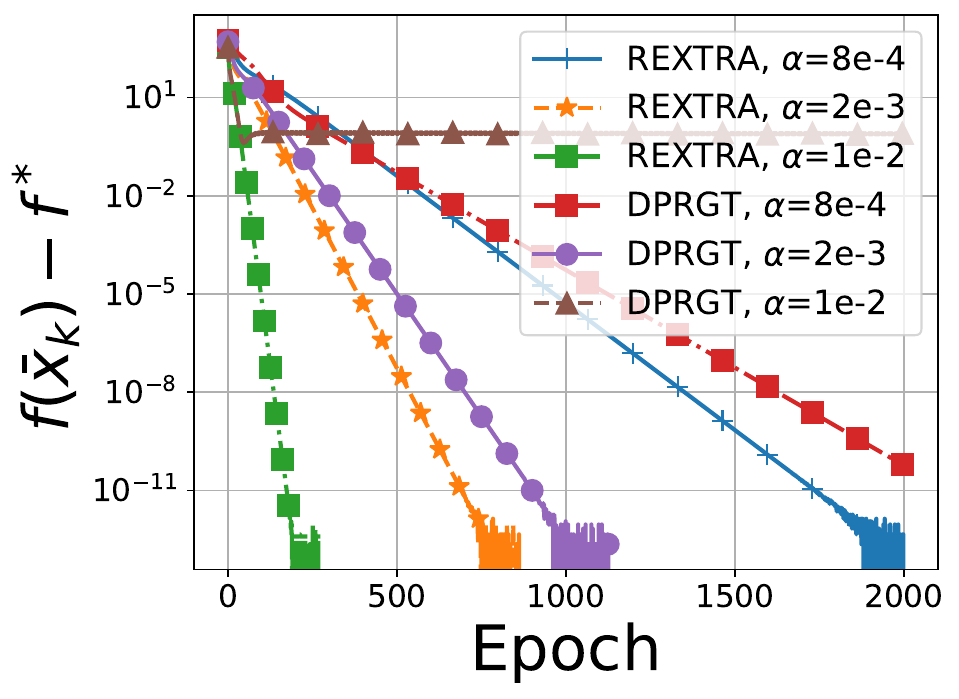}}
    {\includegraphics[width=0.24\textwidth]{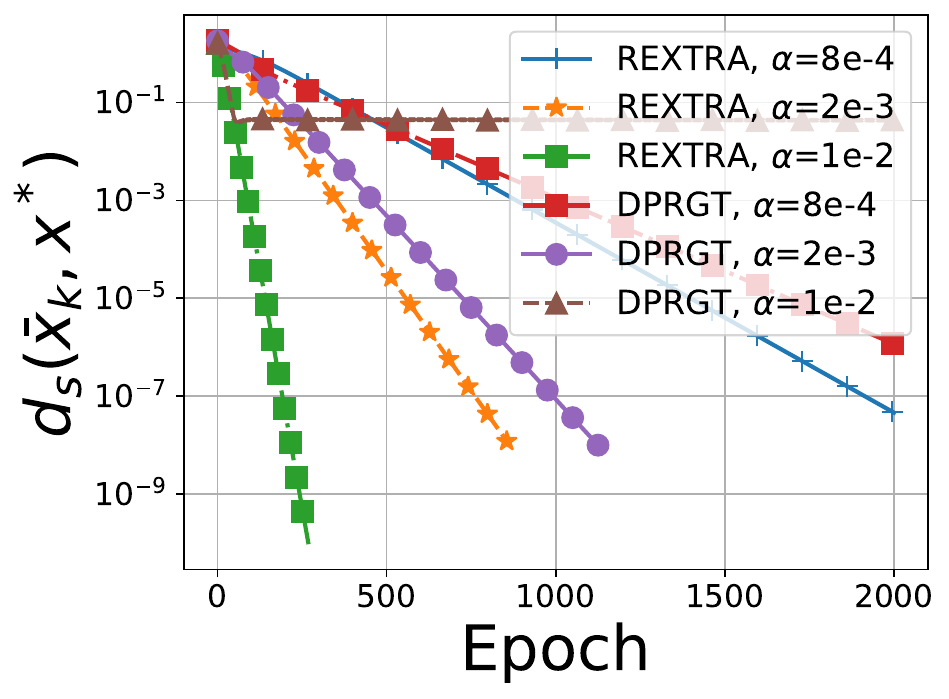} }
    \caption{Results for different stepsizes on the synthetic PCA problem.}
    \label{fig:dd1}
\end{figure}

{\par\medskip\noindent\textbf{MNIST dataset.}}
{To further evaluate the efficiency of REXTRA, we conduct numerical experiments on the MNIST dataset~\cite{lecun1998mnist}. The dataset contains 60,000 handwritten-digit images of size $28\times28$, which are used to generate the local matrices $A_i$. We normalize the pixel values by dividing by 255 and randomly partition the samples among $n=8$ nodes in equal-sized blocks. Consequently, each node holds a local matrix $A_i\in\mathbb R^{(60000/n)\times784}$. We compute the top $r=5$ principal components in ambient dimension $d=784$. As in the preceding experiment, all algorithms use constant stepsizes, and the best stepsize for each method is selected by a grid search over $\{1,2,4,8\}\times\{10^{-6},10^{-5},10^{-4},10^{-3}\}$. Figure~{\ref{fig:mnist1}} shows that REXTRA supports larger stepsizes and outperforms the compared algorithms in convergence versus communication.}

\begin{figure}[htp]
    \centering
    {\includegraphics[width=0.24\textwidth]{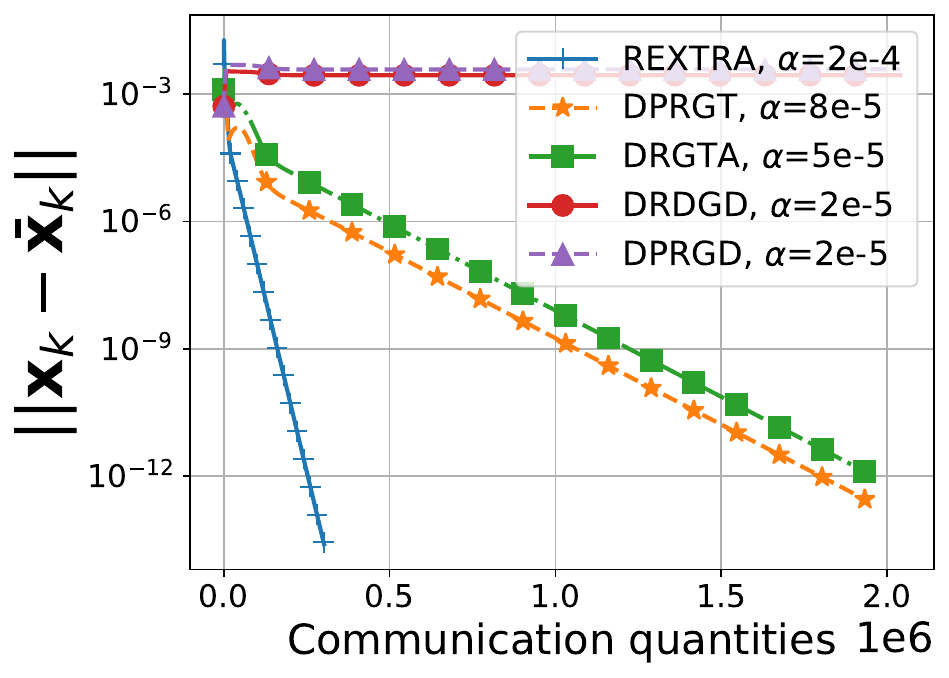}}
    \hfill
    {\includegraphics[width=0.24\textwidth]{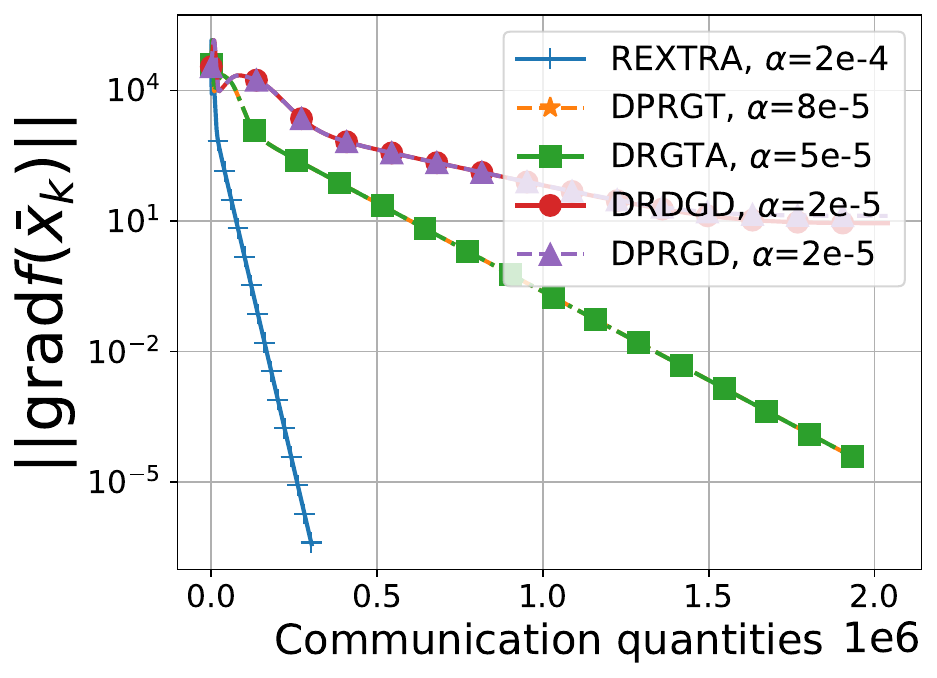}}
    \hfill
    {\includegraphics[width=0.24\textwidth]{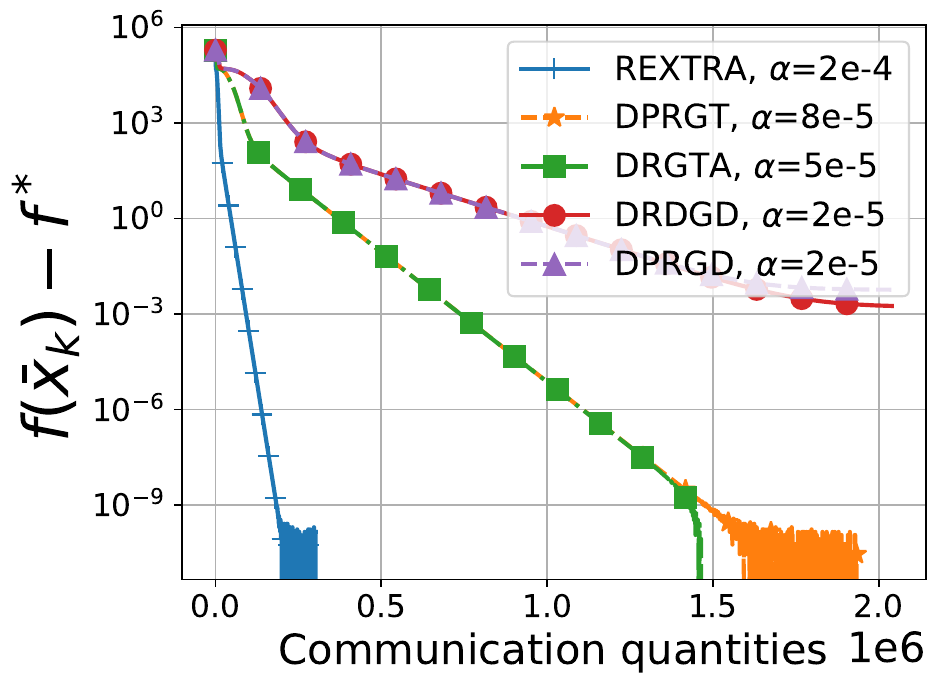}}
    {\includegraphics[width=0.24\textwidth]{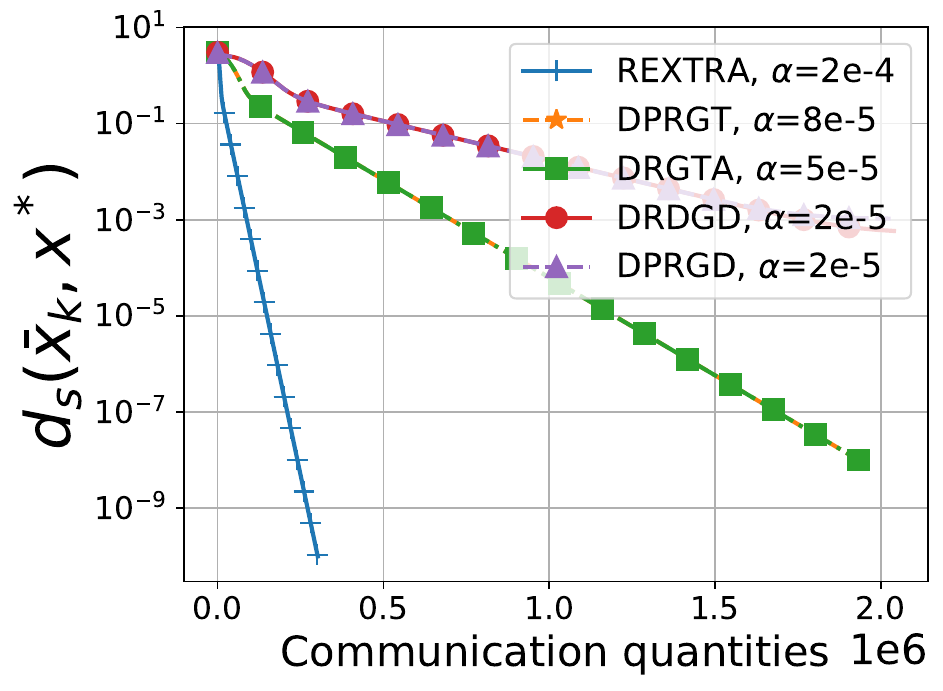}}
    \caption{Results of the tested algorithms for the PCA problem on MNIST, measured by communicated scalar entries.}
    \label{fig:mnist1}
\end{figure}

\subsection{Distributed generalized eigenvalue problems}
We consider the following decentralized generalized eigenvalue problem:

\begin{equation} \label{prob-geig}
\begin{aligned}
\min_{\mathbf{x}}\quad &{\frac{1}{2n}\sum_{i=1}^n
\operatorname{tr}\!\left(x_i^\top A_i^\top A_i x_i\right),} \\
\text{s.t.}\quad &x_1=\cdots=x_n, \\
&x_i^\top B x_i = I_r, \quad\forall i \in [n],
\end{aligned}
\end{equation}
{Here $B\in\mathbb R^{d\times d}$ is symmetric positive definite. Under the consensus constraint, minimizers of \eqref{prob-geig} span the generalized eigenspace associated with the $r$ smallest eigenvalues of the pencil $(\sum_{i=1}^n A_i^\top A_i,B)$. Algorithms for the corresponding centralized generalized eigenvalue problem are well studied; see, e.g., \cite{golub2013matrix}.}
{Generalized eigenvalue formulations, with the smallest- or largest-eigenspace orientation chosen according to the application, arise in kernel supervised PCA \cite{barshan2011supervised}, Fisher discriminant analysis \cite{fisher1936use}, and canonical correlation analysis \cite{chaudhuri2009multi,hotelling1992relations}. Equation~\eqref{prob-geig} uses the smallest-eigenspace orientation.}

{We conduct numerical experiments on the MNIST dataset with $n=8$, using the partitioning procedure in Section~\ref{sec-PCA} to construct the local matrices $A_i$. To generate $B$, we first draw an orthogonal matrix $Q_B\in\mathbb R^{d\times d}$ and then set
$B=Q_B\Lambda Q_B^\top$, where $d=784$ and
$\Lambda:=\operatorname{diag}(1.1^{(j-1)/2})_{j=1}^{d}$; this formula explicitly defines all $d$ positive diagonal entries. In particular, $\kappa_2(B)=1.1^{(d-1)/2}\approx1.60\times10^{16}$ for $d=784$, so this experiment uses a severely ill-conditioned generalized Stiefel constraint. Each node holds the same number of samples. All algorithms use constant stepsizes, and the best stepsize for each method is selected by a grid search over $\{1,2,4,8\}\times\{10^{-6},10^{-5},10^{-4},10^{-3}\}$. As shown in Figure~{\ref{fig:SD3-geig-real}}, REXTRA accepts larger stepsizes and requires less communication than the compared algorithms to reach the same accuracy.}

\begin{figure}[htp]
    \centering
    {\includegraphics[width=0.24\textwidth]{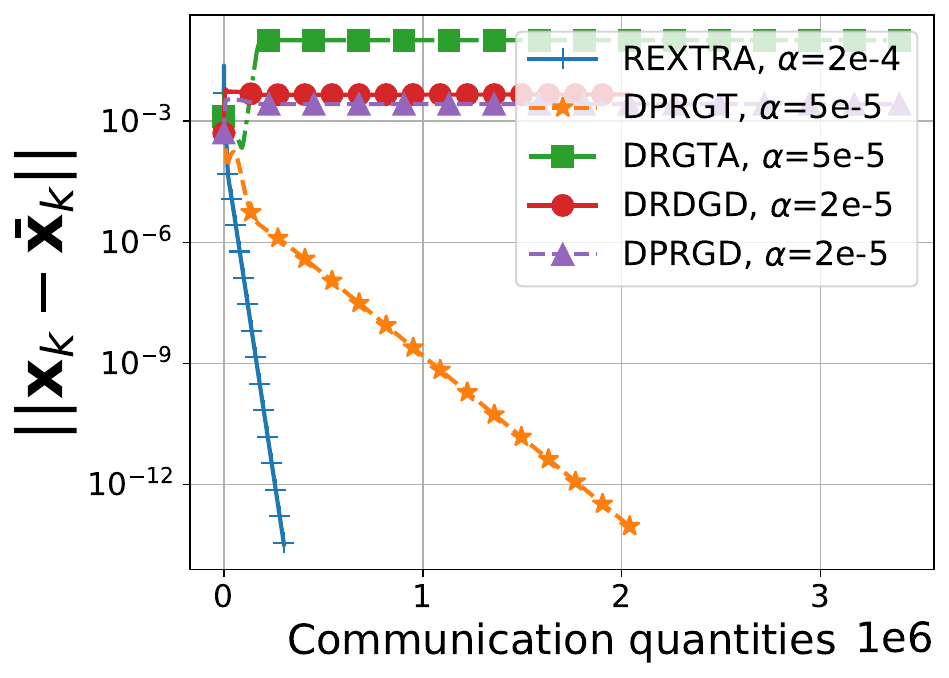}}
    \hfill
    {\includegraphics[width=0.24\textwidth]{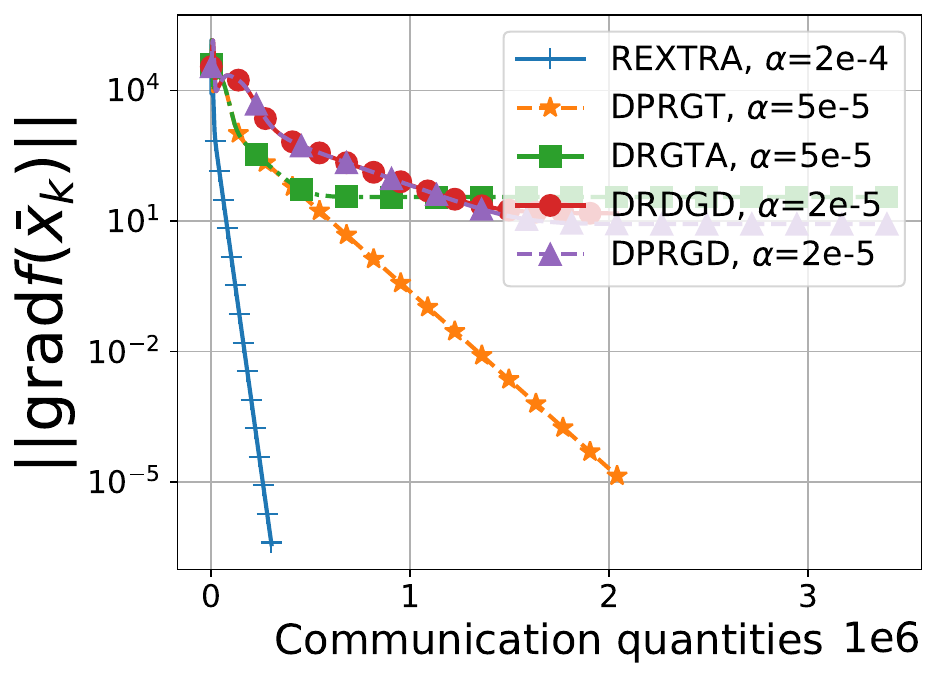}}
    \hfill
    {\includegraphics[width=0.24\textwidth]{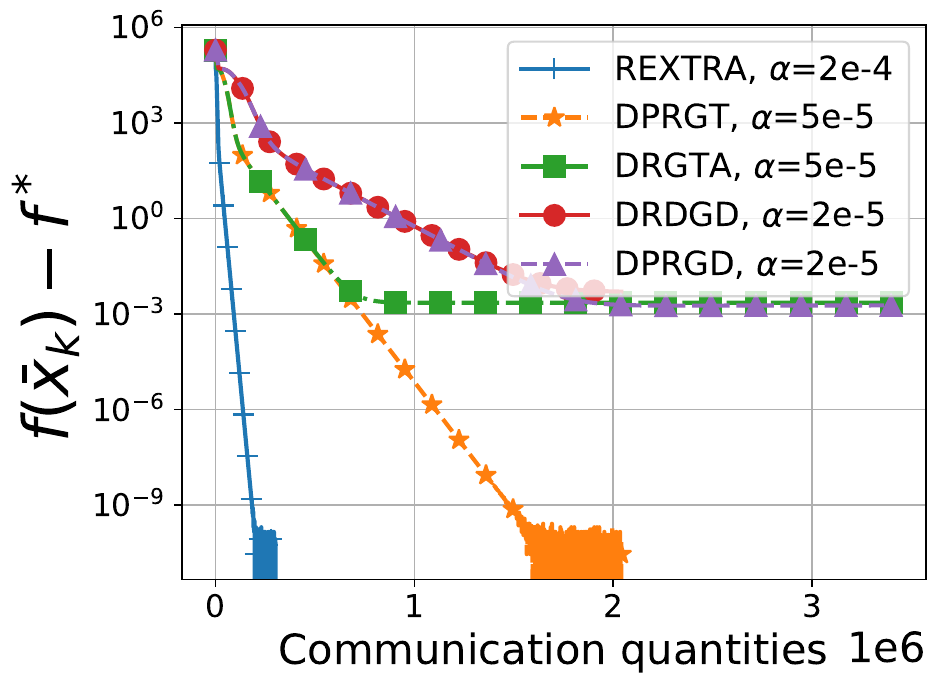}}
    \hfill
    {\includegraphics[width=0.24\textwidth]{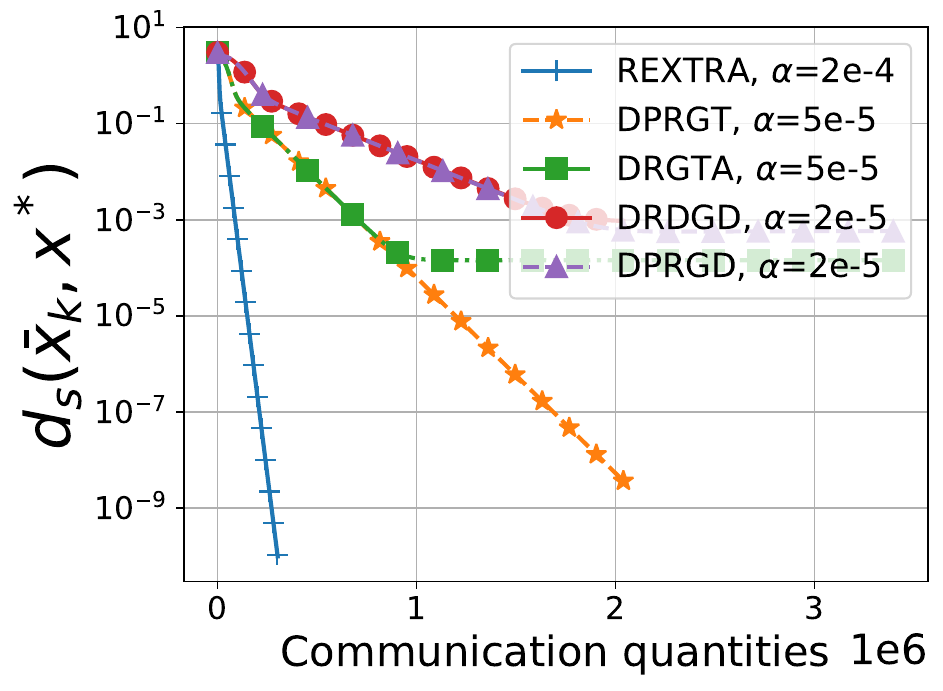}}
    \caption{Results of the tested algorithms for the generalized eigenvalue problem on MNIST, measured by communicated scalar entries.}
    \label{fig:SD3-geig-real}
\end{figure}

\subsection{Decentralized low-rank matrix completion}
{Low-rank matrix completion (LRMC) aims to reconstruct a low-rank matrix $A\in\mathbb R^{d\times T}$ from a subset of its entries. Let $\Omega$ be the observed-index set, and represent each point of $\mathrm{Gr}(d,r)$ by an orthonormal basis matrix $X\in\mathrm{St}(d,r)$. The rank-$r$ LRMC problem can then be written as \eqref{eq:lrmc}.}
\begin{equation}
\label{eq:lrmc}
{\min_{X \in \mathrm{St}(d, r), \, V \in \mathbb{R}^{r \times T}} \frac{1}{2} \left\| \mathcal{P}_\Omega \odot (XV - A) \right\|_F^2,}
\end{equation}
{Here $\mathrm{St}(d,r)$ is the Stiefel manifold of orthonormal basis matrices, $\mathrm{Gr}(d,r)$ is the Grassmann manifold of $r$-dimensional subspaces of $\mathbb R^d$, and $\odot$ denotes the Hadamard product. The binary mask $\mathcal P_\Omega$ is defined by $(\mathcal P_\Omega)_{ij}=1$ for $(i,j)\in\Omega$ and $(\mathcal P_\Omega)_{ij}=0$ otherwise.}

{The objective in \eqref{eq:lrmc} is invariant under $(X,V)\mapsto(XQ,Q^\top V)$ for every orthogonal $Q\in\mathbb R^{r\times r}$. {After minimizing over $V$, the reduced objective is therefore invariant under $X\mapsto XQ$ and defines an optimization problem on $\mathrm{Gr}(d,r)$.} In the decentralized setting, the partially observed matrix $\mathcal P_\Omega\odot A$ is partitioned columnwise into $n$ equal-sized blocks $A_1,\ldots,A_n$. We use Stiefel representatives and impose consensus on those representatives, as in decentralized PCA; {this does not restrict the attainable subspaces because every common subspace admits a common orthonormal-basis representative.} See \cite{absil2009optimization,hu2022riemannian} for the relation between quotient-manifold optimization and optimization on the total space. The resulting decentralized problem is \eqref{eq:decentralized_lrmc}.}
\begin{equation}
\label{eq:decentralized_lrmc}
\begin{aligned}
\min_{X_1,\ldots,X_n} \quad & \frac{1}{2n} \sum_{i=1}^{n} \left\| \mathcal{P}_{\Omega_i} \odot (X_i V_i(X_i) - A_i) \right\|_F^2,\\
\text{s.t.} \quad & X_1 = \cdots = X_n, \; X_i \in \mathrm{St}(d, r), \;\; \forall i \in [n],
\end{aligned}
\end{equation}
{Here $\Omega_i$ is the set of observed indices in the local matrix $A_i$, and the local coefficient matrix $V_i(X)$ is defined by \eqref{eq:lrmc-local-coefficient-split}:}
\begin{equation}\label{eq:lrmc-local-coefficient-split}
{
V_i(X)\in\operatorname*{argmin}_{V\in\mathbb R^{r\times T_i}}
\left\|\mathcal P_{\Omega_i}\odot(XV-A_i)\right\|_F^2,
}
\end{equation}
{with $T_i$ denoting the number of columns assigned to node $i$; if the least-squares minimizer is not unique, $V_i(X)$ denotes the minimizer with minimum Frobenius norm. More precisely, after vectorizing the observed entries column by column, each local least-squares problem has a design operator determined by the corresponding observed rows of $X$. On any neighborhood where these design operators have constant rank, their Moore--Penrose pseudoinverses, the minimum-norm solutions $V_i(X)$, and the associated residual projectors depend smoothly on $X$. Consequently, {the reduced local objective has a Lipschitz continuous Riemannian gradient on every compact subset of such a regular neighborhood.}
For the numerical experiments, we independently draw random matrices $L_{\rm true}\in\mathbb R^{d\times r}$ and $R_{\rm true}\in\mathbb R^{r\times T}$ with standard Gaussian entries and set $A:=L_{\rm true}R_{\rm true}\in\mathbb R^{d\times T}$. We sample each index independently with probability $\mu:=r(d+T-r)/(dT)$ to obtain the observation set $\Omega$. We set $T=1000$, $d=100$, and $r=5$, and partition the columns equally among $n=8$ nodes.}

{We use a ring graph for inter-node communication. All algorithms use constant stepsizes, and the best stepsize for each method is selected by a grid search over $\{1,2,5,8\}\times\{10^{-4},10^{-3},10^{-2}\}$. We set the maximum number of epochs to 800 and terminate early if $\|\grad F(\bar{x}_k)\|<10^{-8}$. Figure~{\ref{fig:lrmc1}} shows that REXTRA attains the target accuracy with the lowest total communication cost among the tested algorithms, owing to its low per-iteration communication volume and its ability to use larger stepsizes; the latter behavior is illustrated in Figure~{\ref{fig:dd3}}.}

\begin{figure}[H]
    \centering
      {\includegraphics[width=0.32\textwidth]{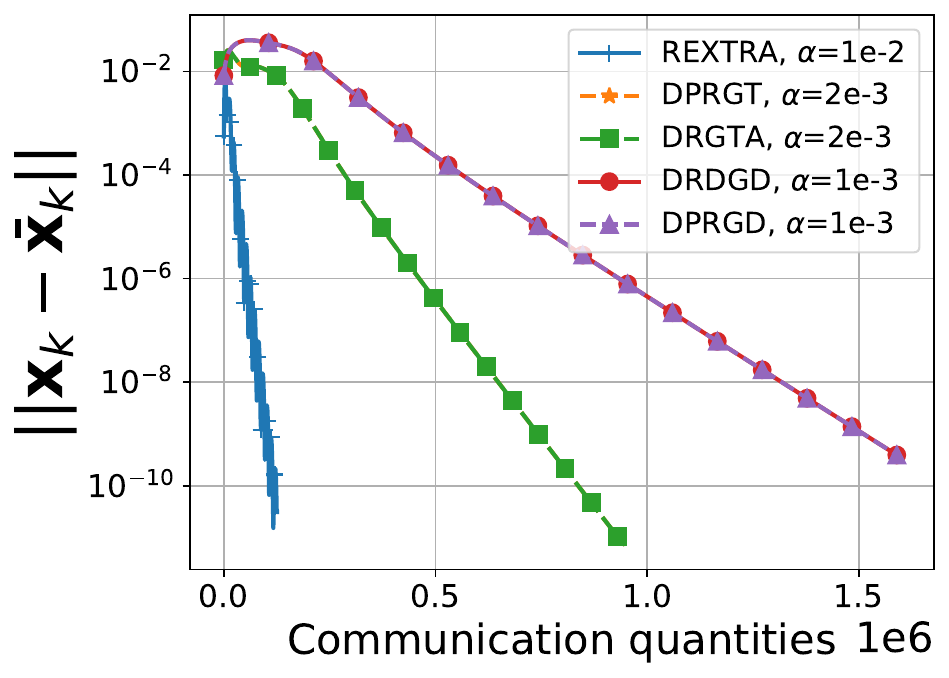}}
      \hfill
    {\includegraphics[width=0.32\textwidth]{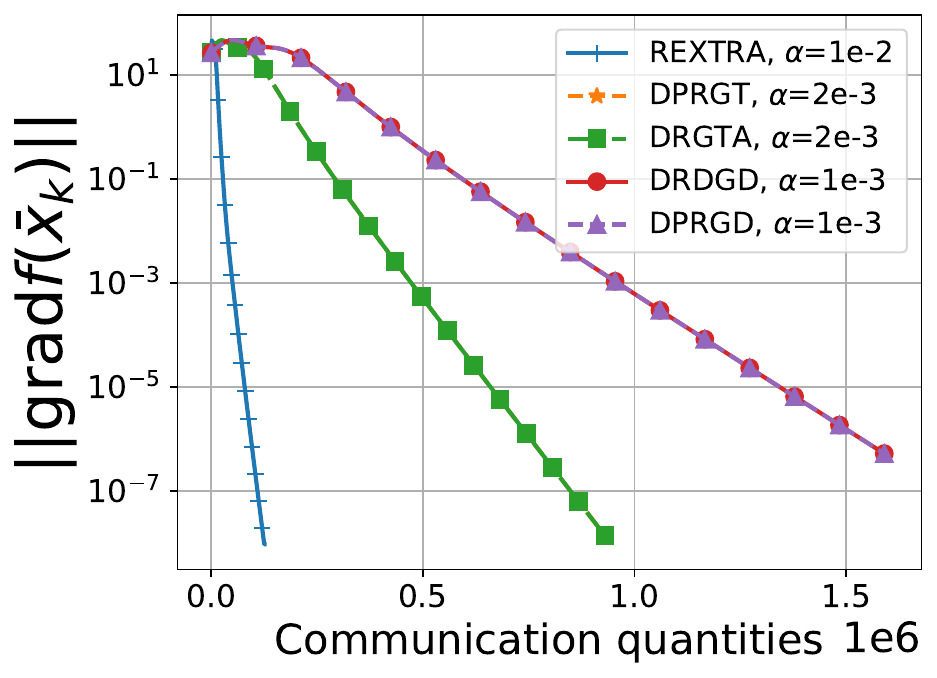}}
    \hfill
    {\includegraphics[width=0.32\textwidth]{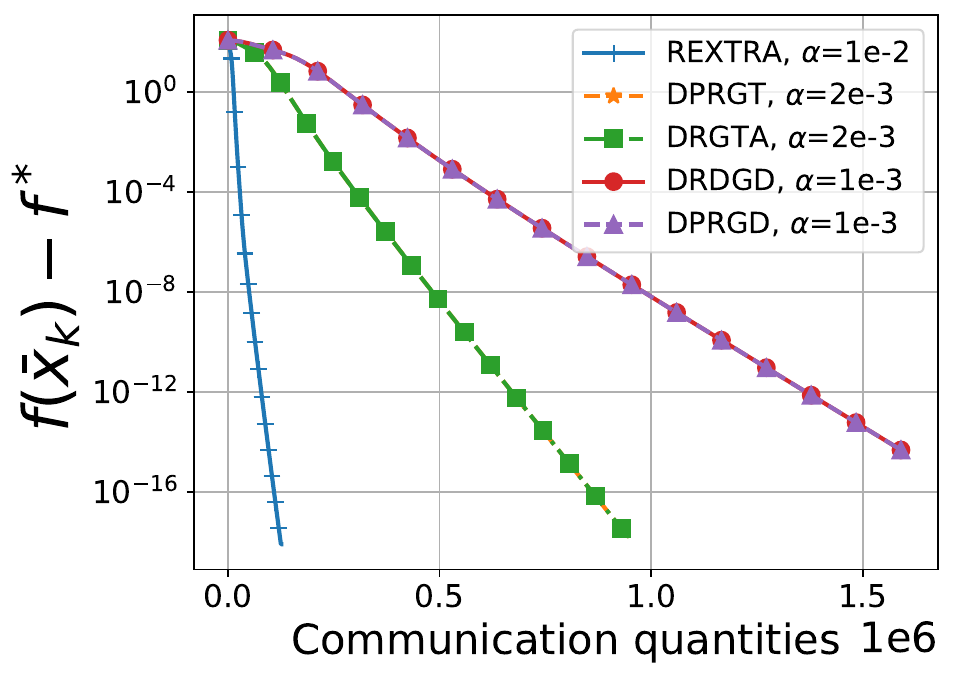}}
    \caption{Results of the tested algorithms for LRMC, measured by communicated scalar entries.}
    \label{fig:lrmc1}
\end{figure}

\begin{figure}[H]
    \centering
      {\includegraphics[width=0.32\textwidth]{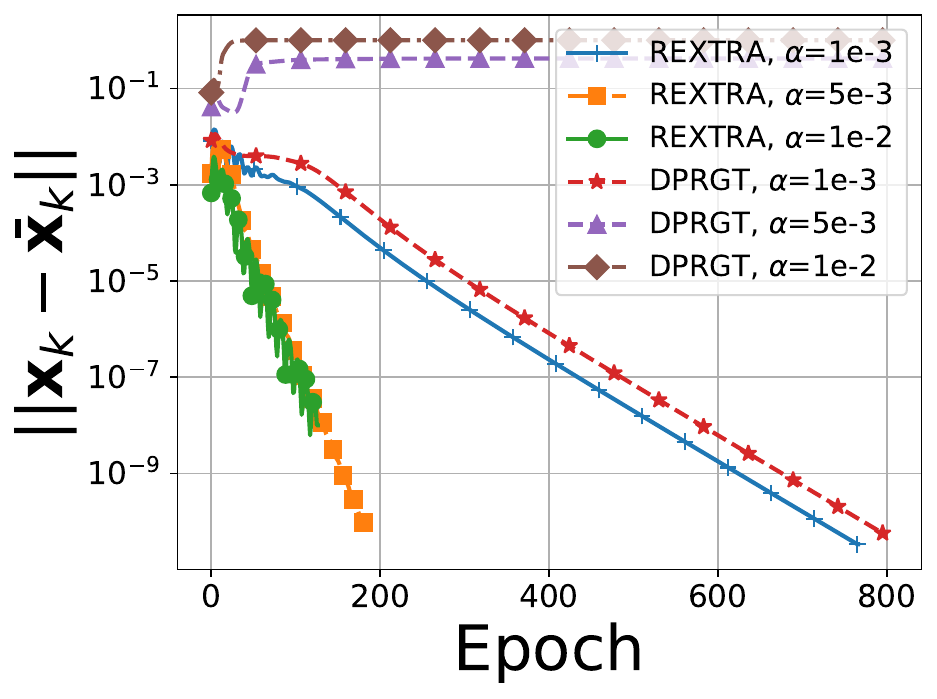}}
      \hfill
    {\includegraphics[width=0.32\textwidth]{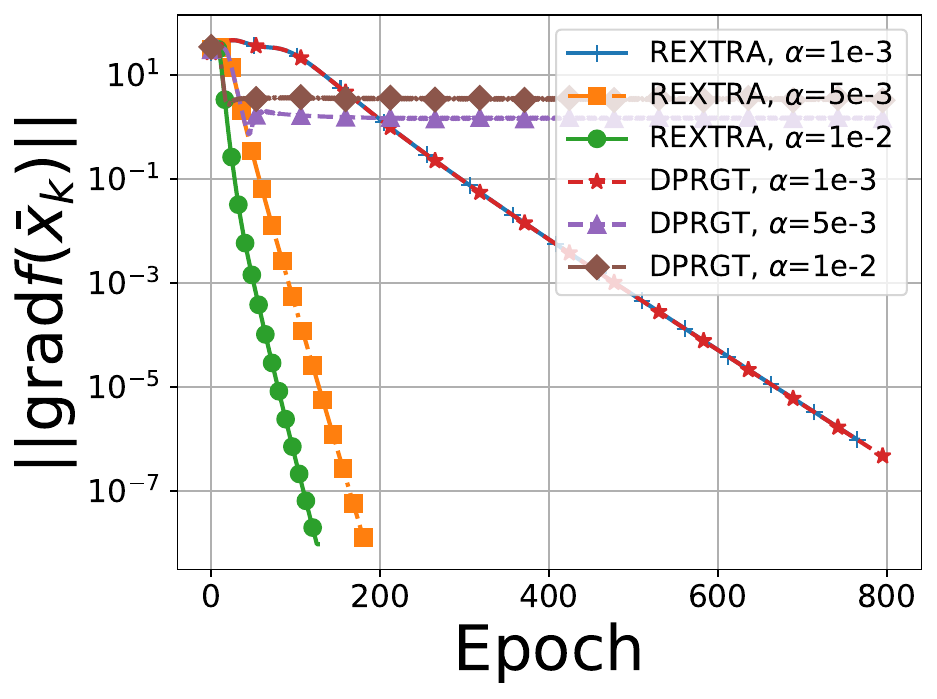}}
    \hfill
    {\includegraphics[width=0.32\textwidth]{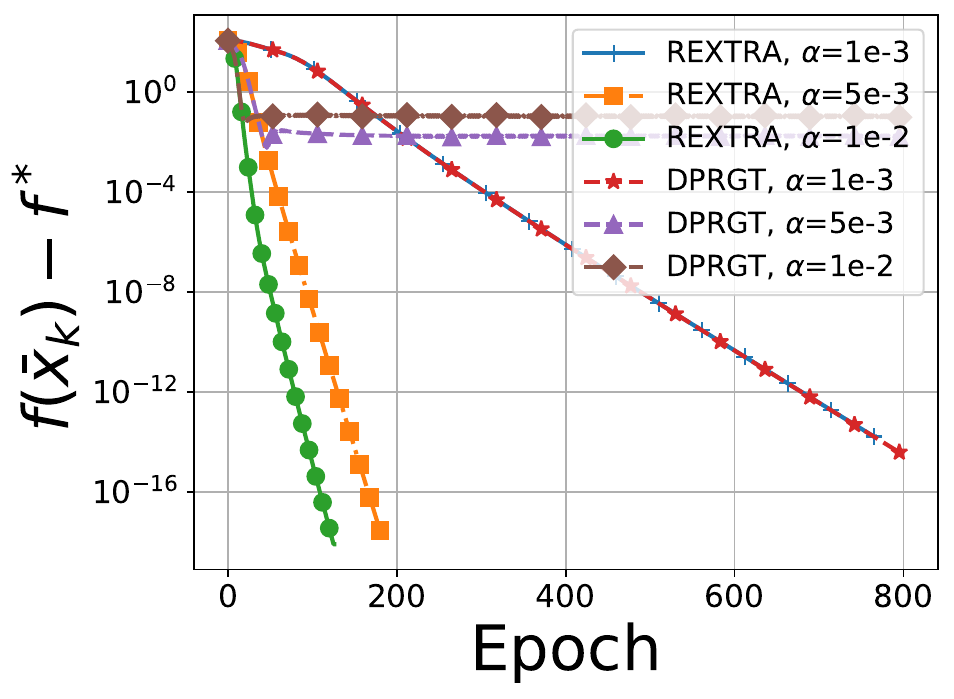}}
    \caption{Results for different stepsizes on the LRMC problem.}
    \label{fig:dd3}
\end{figure}

\section{Conclusion}
{In this paper, we propose projection-based REXTRA and an Algorithm-1-structured retraction-based counterpart for smooth nonconvex optimization over compact embedded submanifolds. The two methods communicate one local-variable-sized message per optimization iteration. The projection method has a local stationarity--consensus guarantee and an $\mathcal O(1/k)$ best-iterate bound. For the retraction method, the uniformly second-order projection--retraction comparison extends the local convergence guarantees and the $\mathcal O(1/k)$ best-iterate rate under the stated initialization conditions. A consensual feasible initialization satisfies the required initial-consensus condition for either method.}

\bibliographystyle{siamplain}
\bibliography{ref}
\end{document}